\documentclass[11pt]{article}
\usepackage{amsfonts}
\usepackage{graphics}
\usepackage{indentfirst}
\usepackage{color}
\usepackage{cite}
\usepackage{latexsym}
\usepackage[paper=a4paper, left=1.6cm, right=1.6cm, top=1.8cm, bottom=1.6cm, headheight=5.5pt, footskip=0.8cm, footnotesep=0.8cm, centering, includefoot]{geometry}
\usepackage{amsmath}
\allowdisplaybreaks
\usepackage{amssymb}
\usepackage[colorlinks, linkcolor=red]{hyperref}
\hypersetup{urlcolor=red, citecolor=red}
\usepackage[dvips]{epsfig}
\usepackage{amscd}

\usepackage{amsthm}
\usepackage{mathrsfs}
\usepackage{verbatim}

\DeclareMathOperator{\divv}{div}
\DeclareMathOperator{\curl}{curl}
\DeclareMathOperator{\loc}{loc}
\allowdisplaybreaks

\begin{document}
\title{Global weak solutions to the isentropic compressible Navier--Stokes equations with vacuum and unbounded density in a half-plane under Dirichlet boundary conditions
\thanks{Wang's research was partially supported by Postgraduate Research and Innovation Project of Chongqing (No. CYB260103). Zhong's research was partially supported by National Natural Science Foundation of China (No. 12371227) and Fundamental Research Funds for the Central Universities (No. SWU--KU24001).
}
}

\author{Shuai Wang,\ Xin Zhong {\thanks{Corresponding author. E-mail addresses: swang238@163.com (S. Wang),
xzhong1014@amss.ac.cn (X. Zhong).}}\date{}\\
\footnotesize
School of Mathematics and Statistics, Southwest University, Chongqing 400715, P. R. China}

\maketitle
\newtheorem{theorem}{Theorem}[section]
\newtheorem{definition}{Definition}[section]
\newtheorem{lemma}{Lemma}[section]
\newtheorem{proposition}{Proposition}[section]
\newtheorem{corollary}{Corollary}[section]
\newtheorem{remark}{Remark}[section]
\renewcommand{\theequation}{\thesection.\arabic{equation}}
\catcode`@=11 \@addtoreset{equation}{section} \catcode`@=12
\maketitle{}

\begin{abstract}
We establish the global existence of a class of weak solutions to the isentropic compressible Navier--Stokes equations in a half-plane with Dirichlet boundary conditions, allowing for vacuum both in the interior and at infinity, under a suitably small initial total energy. The solutions constructed here admit unbounded densities and lie in an intermediate regularity regime between the finite-energy weak solutions of Lions--Feireisl and the framework of Hoff. This result generalizes previous works of Hoff (Comm. Pure Appl. Math. 55 (2002), pp. 1365--1407) and Perepelitsa (Arch. Ration. Mech. Anal. 212 (2014), pp. 709--726) concerning discontinuous solutions by allowing vacuum states and unbounded density. Our analysis relies on the Green function method and new estimates involving the specific structure of the equations and the geometry of the half-plane. To the best of our knowledge, this is the first result concerning global weak solutions within Hoff's framework on an unbounded domain that simultaneously accommodates Dirichlet boundary conditions and far-field vacuum. The intermediate-regularity class developed here may be viewed as a natural extension of Hoff's theory, precisely tailored to overcome the two corresponding obstructions: the lack of global space-time control of the effective viscous flux arising from far-field vacuum and the absence of boundary-induced regularity gains in the no-slip setting.
\end{abstract}




\tableofcontents

\section{Introduction}
\subsection{Background and motivation}
We study the isentropic compressible Navier--Stokes equations in the half-plane $\mathbb{R}^2_+=\{\mathbf{x}\in\mathbb{R}^2:x_2>0\}$:
\begin{align}\label{a1}
\begin{cases}
\rho_t+\divv(\rho\mathbf{u})=0,\\
(\rho\mathbf{u})_t+\divv(\rho\mathbf{u}\otimes\mathbf{u})+\nabla P=
\mu\Delta\mathbf{u}+(\mu+\lambda)\nabla\divv\mathbf{u}
\end{cases}
\end{align}
with the given initial data and far-field conditions
\begin{gather}\label{a2}
(\rho,\rho\mathbf{u})|_{t=0}=(\rho_0,\rho_0\mathbf{u}_0)(\mathbf{x}),~~ \mathbf{x}\in\mathbb{R}^2_+,\\
(\rho_0,\rho_0\mathbf{u}_0)(\mathbf{x})\rightarrow(0,\mathbf{0}) \ \ \text{as}\ |\mathbf{x}|\rightarrow\infty,\ \ \mathbf{x}\in\mathbb{R}^2_+,\label{a3}
\end{gather}
and Dirichlet boundary conditions
\begin{equation}\label{a4}
\mathbf{u}(\mathbf{x},t)=\mathbf{0}, \ \ x_2=0,\ t>0,
\end{equation}
where $\mathbf{x}=(x_1,x_2)$ is the spatial coordinate, and $t\geq0$ is the time. The unknowns $\rho$, $\mathbf{u}=(u^1,u^2)$, and $P=P(\rho)=a\rho^\gamma\ (a>0,\gamma>1)$ are the fluid density, velocity, and pressure, respectively. The constants $\mu$ and $\lambda$ represent the shear viscosity and bulk viscosity of the fluid, respectively, satisfying the physical restrictions
\begin{equation*}
\mu>0,\ \ \ \mu+\lambda\geq0.
\end{equation*}

The global well-posedness of the multi-dimensional isentropic compressible Navier--Stokes equations has been studied extensively, motivated by both the rich mathematical structure of the system and its fundamental physical relevance. In the following, we restrict our discussion to those works that are closely related to the present work.
Matsumura and Nishida \cite{MN80,MN83} first obtained global classical solutions for the three-dimensional initial-boundary value problem, under the assumption that the initial data are sufficiently close to constant equilibrium states in $H^3$. Later on, Danchin \cite{Da00} established the global existence of strong solutions for small initial data
$(\rho_0,\mathbf{u}_0)$ in the scaling-critical Besov spaces
$\dot{B}^{\frac{N}{2}}_{2,1}(\mathbb{R}^N)\times\dot{B}^{\frac{N}{2}-1}_{2,1}(\mathbb{R}^N)\ (N\geq2)$. This critical framework was subsequently extended to the $L^p$ setting in \cite{CD10,CCZ10,H11}. It should be noted that all of the above results require the initial density to be strictly positive, thereby excluding the presence of vacuum states.

To advance the theory of compressible flows, increasing attention has been paid to regimes that admit vacuum. In this setting, the degeneracy of the governing equations, together with the evolution of vacuum boundaries, places the problem beyond the scope of standard degenerate PDE theory and necessitates the development of new analytical approaches. The seminal breakthrough is due to P.-L. Lions \cite{PL98}, who established finite-energy weak solutions for the isentropic system with vacuum, satisfying the global energy inequality and characterized by densities in $L^\infty_tL^\gamma_x$ with $\gamma \geq \frac{3N}{N+2}$ for $N=2,3$. The key for the construction was to build strong compactness of the density sequence via developing the renormalization techniques and the effective viscous flux arguments. This foundation has inspired significant extensions. Feireisl and collaborators \cite{F04,FNP01} relaxed the restriction on the adiabatic exponent to $\gamma > \frac{N}{2}$ with the aid of oscillation defect measures. Under assumptions of spherical symmetry or axisymmetry,
Jiang and Zhang \cite{JZ01,JZ03} established the existence of global weak solutions for any $\gamma>1$ by developing refined estimates for the effective viscous flux, which in turn enhance the integrability properties of the density.
Recently, Bresch and Jabin \cite{BJ18} extended the Lions--Feireisl framework to incorporate more complex physics, including non-monotone pressure laws and anisotropic viscous stress tensors. Despite these advances, the question of uniqueness remains a fundamental and largely open problem in the theory of weak solutions.

Parallel to the development of weak solutions theory, substantial progress has been made on the global well-posedness of classical solutions in the presence of vacuum. A significant contribution was made by Huang, Li, and Xin \cite{HLX12}, who established the global existence and uniqueness of classical solutions in $\mathbb{R}^3$ with vacuum for initial data with small total energy but possibly large oscillations.
To address the distinct logarithmic decay at infinity in two dimensions, this framework was extended to the whole plane in \cite{LX19} through the development of weighted estimates and a refined decay analysis. More recently, a relaxation of the small-energy assumption was achieved in \cite{HHPZ24} for the Cauchy problem as the adiabatic exponent $\gamma$ approaches $1$. It was shown in \cite{M2} that local smooth 3D solutions with  $\gamma\leq1+2/\sqrt{3}$ explode in finite time: the $L^\infty$ norms of the density and the velocity blow-up. Thus, in some sense, the smallness condition is necessary in order to ensure global well-posedness. It should be emphasized that one of the key issues in \cite{HHPZ24,HLX12,LX19} is to derive time-independent uniform upper bounds of the density. For more studies on the compressible Navier--Stokes equations, we refer the reader to the excellent handbook \cite{GN18} and references contained therein.
Let us point out that although the weighted energy method is effective in handling classical solutions as in \cite{LX19}, it is difficult to deal with the theory of global weak solutions in $\mathbb{R}^2$ with far-field vacuum due to the lack of decay-sensitive compactness for either the density or the pressure, which obstructs a satisfactory treatment of the nonlinear terms.

Returning to the theory of weak solutions, another significant category is formed by Hoff's intermediate weak solutions. In his foundational works \cite{Hoff95,Hoff95*}, Hoff justified the global existence of such solutions in the whole space, under the specific structural assumptions that the initial density $\rho_0$ is close to a constant in $L^\infty$. These solutions possess regularity intermediate between that of the Lions--Feireisl weak solutions and standard strong solutions.
In particular, particle trajectories are well-defined in non-vacuum regions, while density discontinuities may be transported along them (see \cite{Hoff02}).
Building on this Lagrangian viewpoint, and using the elliptic structure provided by the effective viscous flux together with the $L^\infty$ control of the density, Hoff \cite{Hoff06} proved uniqueness and continuous dependence for multidimensional singular weak solutions, thereby establishing well-posedness within this framework.  Hoff and Santos \cite{Hoff08} later developed this Lagrangian structure further and applied it to the propagation of density singularities.
It should be emphasized, however, that Hoff's whole-space theory, while allowing interior vacuum, is perturbative around a strictly positive constant far-field state. Consequently, the genuinely far-field-vacuum regime lies beyond this framework and remains a major unresolved problem.

While the preceding discussion has focused on the Cauchy problem, the initial-boundary value problem for isentropic compressible flows in domains with physically relevant boundary conditions forms a distinct and equally important line of research.
For small-energy initial data, Hoff \cite{Hoff05} constructed global weak solutions allowing vacuum in the half-space $\mathbb{R}_+^3$ under Navier-slip boundary conditions.
Global classical solutions were later established for the same type of slip boundary condition in the half-space \cite{D12} and in bounded domains \cite{CL23}.
A crucial analytical advantage in these settings is that slip boundary conditions provide explicit boundary information for the effective viscous flux $F=(2\mu+\lambda)\divv\mathbf{u}-P(\rho)$, a quantity central to the derivation of enhanced regularity. By contrast, despite its physical relevance, progress under no-slip boundary conditions has been much more limited. In this regard, Perepelitsa \cite{P14} extended Hoff's framework to Dirichlet (no-slip) boundaries under the assumptions that the initial density is strictly positive and belongs to a boundary-adapted H\"older space defined by reflection, and that the viscosity ratio $\mu/(\mu+\lambda)$ is suitably small. However, the approach in \cite{P14} appears difficult to extend to genuine no-slip boundaries in the absence of additional structural assumptions on the density near the boundary, particularly in the presence of vacuum.

It is worth noting that the propagation of an $L^\infty$ bound for the density relies crucially on the algebraic structure of the effective viscous flux (see \cite{Hoff02}).
\textit{To overcome the intrinsic obstacles associated with Dirichlet boundary conditions, it may therefore be necessary to move beyond the classical isentropic framework and consider regimes in which this structural quantity no longer plays a central role}.
In this spirit, Bresch and Burtea \cite{BB23} established global weak solutions on the torus $\mathbb{T}^N$ for compressible flows with an anisotropic viscous stress tensor, requiring only an
$L^p$ bound on the density instead of the stronger $L^\infty$ condition.
Their strategy bypasses the need for regularity gain from the effective viscous flux and thus suggests a potential pathway for addressing the isentropic no-slip problem.
Nevertheless, the presence of physical boundaries introduces additional difficulties: density irregularities near the boundary may interact with the boundary in a singular way, potentially leading to loss of transversality and cusp-like tangencies of transported discontinuity interfaces \cite{HP09,HP12}.
On the other hand, the existence of a unique global strong solution in $\mathbb{R}^3$ with far-field vacuum was recently obtained in \cite{W25} under a smallness assumption on a scaling-invariant quantity, with the analysis relying on the propagation of an additional $L^{2\gamma}$ bound for the density.
\textit{From a functional analytic perspective, an intermediate $L^p$ regularity with $\gamma < p < \infty$ places the pressure $P(\rho)=a\rho^\gamma$ in a Lebesgue space that lies strictly between the $L^1$ setting of the Lions--Feireisl theory and the $L^\infty$ framework of Hoff}.
This level of spatial integrability provides access to hidden decay information at infinity, which may be instrumental in the treatment of far-field vacuum.

Motivated by the above considerations, we aim to construct global weak solutions in an $L^p$ framework for the isentropic compressible Navier--Stokes equations in the half-plane under Dirichlet boundary conditions, allowing vacuum both in the interior and at infinity.
The main difficulty lies not only in each aspect separately, but also in their strong interaction. Under Dirichlet boundary conditions, key analytical tools, such as the effective viscous flux mechanism and vorticity estimates, become significantly less effective. Meanwhile, the weak solution framework provides no effective decay control at infinity.
This combined loss of boundary structure and far-field control renders many classical approaches inapplicable and necessitates a unified treatment of both effects.
To this end, our approach is twofold: we control the time growth in the far-field analysis by exploiting the $L^q$ $(1<q<\infty)$ integrability of the pressure $P(\rho)$, and we decompose the effective viscous flux via a Green function method, which allows the Dirichlet boundary condition to be incorporated into the analytical framework and enables us to recover the required estimates.

\subsection{Main result}

Before stating the main result, we collect some notation that will be used throughout the paper.
Unless otherwise specified, all function spaces are taken over $\mathbb{R}^2_+$,
and the same notation is used for both scalar-valued and vector-valued functions.
For $1\leq p\leq\infty$ and every integer $k\geq0$, we set
\[
 L^p=L^p(\mathbb{R}^2_+),\quad
 W^{k,p}=W^{k,p}(\mathbb{R}^2_+),\quad
 H^k=W^{k,2}(\mathbb{R}^2_+),\quad
 D_0^1(\mathbb{R}^2_+)=\overline{C_c^\infty(\mathbb{R}^2_+)}^{\|\nabla\cdot\|_{L^2(\mathbb{R}^2_+)}}.
\]

Throughout the paper, $C$ denotes a generic positive constant which may change from line to line.
When its dependence on certain quantities needs to be emphasized, we write, for example, $C(f)$.
The symbol $\Box$ is used to mark the end of a proof, and $a\triangleq b$ means that $a$ is defined to be $b$.

Let $B_R$ denote the open ball of radius $R$ centered at the origin. For simplicity, we write
\begin{align*}
\widetilde{B}_R=B_R\cap\mathbb{R}^2_+, \ \
\int f \mathrm{d}\mathbf{x}=\int_{\mathbb{R}^2_+} f \mathrm{d}\mathbf{x}, \ \ f_i=\partial_if=\frac{\partial f}{\partial x_i}.
\end{align*}
For two $n\times n$ matrices $A=\{a_{ij}\}$ and $B=\{b_{ij}\}$, the symbol $A: B$ represents the trace of the matrix product $AB^\top$, i.e.,
\begin{equation*}
A:B\triangleq\operatorname{tr}(AB^\top)=\sum_{i,j=1}^na_{ij}b_{ij}.
\end{equation*}

The initial total energy is defined as
\begin{equation}\label{1.5}
C_0\triangleq\int_{\mathbb{R}^2_+}
\bigg(\frac{1}{2}\rho_0|\mathbf{u}_0|^2+\frac{1}{\gamma-1}P(\rho_0)\bigg)\mathrm{d}\mathbf{x}.
\end{equation}
Moreover, we denote by
\begin{align}\label{1.6}
\begin{cases}
\dot{f}\triangleq f_t+(\mathbf{u}\cdot\nabla) f,\\
F\triangleq(2\mu+\lambda)\divv\mathbf{u}-P(\rho),\\
\curl\mathbf{u}\triangleq-\nabla^\bot\cdot\mathbf{u}=-\partial_2u^1+\partial_1u^2,
\end{cases}
\end{align}
which represent the material derivative of $f$, the effective viscous flux, and the vorticity, respectively.

We recall the definition of weak solutions to the problem \eqref{a1}--\eqref{a4} in the sense of \cite{Hoff05}.
\begin{definition}\label{d1.1}
A pair $(\rho, \mathbf{u})$ is said to be a weak solution to the problem \eqref{a1}--\eqref{a4} provided that
\begin{equation*}
  (\rho, \rho\mathbf{u})\in C([0,\infty);H^{-1}(\mathbb{R}^2_+)),\ \
  \nabla\mathbf{u}\in L^2(\mathbb{R}^2_+\times(0,\infty))
\end{equation*}
with $(\rho,\mathbf{u})|_{t=0}=(\rho_0,\mathbf{u}_0)$. Moreover,
for any $t_2\geq t_1\geq 0$ and test function $(\phi,\boldsymbol\psi)(\mathbf{x},t)\in C^1\big(\overline{\mathbb{R}^2_+}\times[t_1,t_2]\big)$, with uniformly bounded
support in $\mathbf{x}$ for $t\in[t_1,t_2]$ and satisfying $\boldsymbol\psi|_{\partial\mathbb{R}^2_+}=\mathbf 0$, the following identities hold\footnote{Throughout this paper, we will use the Einstein summation over repeated indices convention.}:
\begin{align*}
\int_{\mathbb{R}^2_+}\rho(\mathbf{x},\cdot)\phi(\mathbf{x},\cdot)
\mathrm{d}\mathbf{x}\Big|_{t_1}^{t_2}
&=\int_{t_1}^{t_2}\int_{\mathbb{R}^2_+}\big(\rho\phi_t+
\rho\mathbf{u}\cdot\nabla\phi\big)\mathrm{d}\mathbf{x}\mathrm{d}t,\\
\int_{\mathbb{R}^2_+}(\rho\mathbf{u}\cdot\boldsymbol{\psi})
(\mathbf{x},\cdot)\mathrm{d}\mathbf{x}\Big|_{t_1}^{t_2}
&=\int_{t_1}^{t_2}
\int_{\mathbb{R}^2_+}\big(\rho\mathbf{u}\cdot\boldsymbol{\psi}_t+
\rho u^i\mathbf{u}\cdot\partial_i\boldsymbol{\psi}+P\divv\boldsymbol{\psi}\big)
\mathrm{d}\mathbf{x}\mathrm{d}t\notag\\
&\quad -\int_{t_1}^{t_2}
\int_{\mathbb{R}^2_+}\big(\mu\nabla\mathbf{u}:\nabla\boldsymbol{\psi}+
(\mu+\lambda)\divv\mathbf{u}\divv\boldsymbol{\psi}\big)\mathrm{d}\mathbf{x}\mathrm{d}t.
\end{align*}
\end{definition}

Given $\alpha\in(1,2)$, assume that the initial data $(\rho_0,\mathbf{u}_0)$ satisfies
\begin{equation}\label{1.12}
  0\leq\rho_0\in L^{\theta}, \ \ \bar{x}^\alpha\rho_0\in L^1, \ \ \mathbf{u}_0\in D^1_0,\ \
  \rho_0|\mathbf{u}_0|^2\in L^1,
\end{equation}
where
\begin{equation}\label{1.10}
\theta\triangleq \frac{4\gamma(2\alpha+1)}{\alpha-1}\in(20\gamma,\infty),
\ \ \bar{x}\triangleq\big(e+|\mathbf{x}|^2\big)^{1/2}\log^2\big(e+|\mathbf{x}|^2\big).
\end{equation}
Without loss of generality, let the initial density $\rho_0$ verify
\begin{equation}\label{1.7}
  \int_{\mathbb{R}^2_+}\rho_0\mathrm{d}\mathbf{x}=1,
\end{equation}
which implies that there exists a constant $\eta_0>0$ such that
\begin{equation}\label{1.8}
  \int_{\widetilde{B}_{\eta_0}}\rho_0\mathrm{d}\mathbf{x}\geq\frac{1}{2}\int_{\mathbb{R}^2_+}\rho_0\mathrm{d}\mathbf{x}=\frac{1}{2}.
\end{equation}
Assume further that, for some constants $\hat{\rho}\geq1$ and $M\geq1$,
\begin{align}\label{1.9}
\|\rho_0\|_{L^{\theta}}\leq \hat{\rho}, \qquad
\|\bar{x}^\alpha\rho_0\|_{L^{1}}+\|\nabla\mathbf{u}_0\|_{L^2}\leq M.
\end{align}
In addition, we require the viscosity coefficients to satisfy the following structural condition:
\begin{equation}\label{1.11}
  \frac{3\mu}{3\mu+\lambda}\leq \min \bigg\{\frac{1}{2\max\limits_{j\in\{2,3,4\}}\Lambda(j)},
  \Lambda^{-1}\bigg(\frac{9\alpha+3}{\alpha-1}\bigg)
  \bigg\},
\end{equation}
where $\Lambda(\cdot)$ denotes the positive constant introduced in Lemma \ref{l2.7} in two dimensions.
Indeed, condition \eqref{1.11} is fulfilled whenever $\lambda/\mu$ is sufficiently large.

Now we state our main result on the global existence of weak solutions.

\begin{theorem}\label{t1.1}
Let the assumptions \eqref{1.12}, \eqref{1.7}, \eqref{1.9}, and \eqref{1.11} be satisfied. Then there exists a positive constant $\varepsilon$ depending only on $\alpha, \hat{\rho}, M, a, \gamma, \mu, \lambda$, and $\eta_0$ such that if
\begin{equation}\label{1.13}
  C_0\leq\varepsilon,
\end{equation}
the problem \eqref{a1}--\eqref{a4} admits a global weak solution $(\rho,\mathbf{u})$ in the sense of Definition $\ref{d1.1}$ satisfying, for any $0<T<\infty$,
\begin{equation}\label{1.14}
\begin{cases}
0\leq\rho\in L^\infty(0,T; L^{\theta}(\mathbb{R}^2_+))\cap C([0,T];L^q(\mathbb{R}^2_+)), \quad \text{for any} \ q\in[1,\theta),\\[0,3em]
\bar{x}^\alpha\rho\in L^\infty(0,T;L^{1}(\mathbb{R}^2_+)),\quad
\sqrt{\rho}\mathbf{u}\in L^{\infty}([0,T];L^2(\mathbb{R}^2_+)),\quad
\nabla\mathbf{u}\in L^2(\mathbb{R}^2_+\times(0,T)),\\[0,3em]
\big(\sigma^{\frac{1}{2}}\nabla\mathbf{u},\ \sigma^{\frac{3}{2}}\sqrt{\rho}\dot{\mathbf{u}}\big)
\in L^\infty(0,T;L^{2}(\mathbb{R}^2_+)),\quad
\big(\sigma^{\frac{1}{2}}\sqrt{\rho}\dot{\mathbf{u}},\ \sigma^{\frac{3}{2}}\nabla\dot{\mathbf{u}}\big)\in L^2(\mathbb{R}^2_+\times(0,T)),
\end{cases}
\end{equation}
with $\sigma=\sigma(t)\triangleq\min\{1,t\}$, and
\begin{equation}\label{1.15}
  \inf_{0\leq t\leq T}\int_{\widetilde{B}_{\eta_1(1+t)\log^{\alpha}(e+t)}}\rho(\mathbf{x},t)\mathrm{d}\mathbf{x}\geq\frac14
\end{equation}
for some positive constant $\eta_1$ depending only on $\alpha, \hat{\rho}, M, a, \gamma,\eta_0$, and $\|\sqrt{\rho_0}\mathbf{u}_0\|_{L^2}$.
\end{theorem}

Several remarks are in order.
\begin{remark}
It should be noted that the weak solutions constructed in Theorem \ref{t1.1} occupy an intermediate regularity regime between the Lions--Feireisl finite-energy weak solutions and the weak solutions developed by Hoff \cite{Hoff95,Hoff95*,Hoff02,Hoff05}. Our framework of a half-plane with far-field vacuum and Dirichlet boundary conditions provides significant insight for two extremes: the Cauchy problem with vacuum at infinity and the Dirichlet problem in bounded domains. We leave such interesting problems for our future studies.
\end{remark}

\begin{remark}
The present result may be regarded as a refinement of Perepelitsa's work in \cite{P14}.
On the one hand, the presence of vacuum is allowed both in the interior and at infinity.
On the other hand, compared with the $L^\infty$ density framework in \cite{P14}, the present construction only requires $L^\theta$ integrability of the density and dispenses with the additional H\"older regularity assumption imposed on the initial density near the boundary.
\end{remark}

\begin{remark}
It should be emphasized that the approach introduced by Bresch and Burtea \cite{BB23} relies essentially on the periodic geometry of the torus $\mathbb{T}^N$, which allows them to make use of harmonic analysis tools due to the absence of boundaries, and therefore cannot be directly extended to the setting with physical boundaries. The absence of an $L^\infty$ bound for the density constitutes a fundamental analytical difference in the treatment of the nonlinear terms; see Subsection \ref{sec1.3} for details. Conversely, the intermediate-regularity viewpoint adopted here may serve as a useful guideline for more complex models, including compressible non-Newtonian flows, where pointwise density bounds are out of reach and integrability-based structural estimates become essential.
\end{remark}

\begin{remark}
The index $\theta$ in \eqref{1.10} is a technical parameter determined by weighted inequalities in Lemma \ref{l2.4} and large-time analysis in \eqref{3.57}, and may be relaxed with a more refined argument; see Subsection \ref{sec1.3} for details.
The assumption \eqref{1.11} on viscosity is introduced solely to handle the difficulties arising from the no-slip boundary.
In addition, it should be mentioned that it is not clear whether the regularity assumption on the initial velocity in Theorem \ref{t1.1} can be generalized to $H^s$ for some $0<s<1$, which has been achieved in the whole plane case \cite{Hoff02} as the initial density is bounded from above and below (among other conditions).
\end{remark}

\begin{remark}
The integrability advantage of the present framework is largely captured by the estimate
\begin{equation*}
\|\rho\dot{\mathbf{u}}\|_{L^{p}}
\leq
\|\sqrt{\rho}\|_{L^\frac{2p}{2-p}}\,
\|\sqrt{\rho}\dot{\mathbf{u}}\|_{L^2},
\  \ \ 1<p<2,
\end{equation*}
which provides a fundamental bridge between the weighted energy structure and the nonlinear estimates, and plays a crucial role in handling both boundary effects and the behavior at infinity; see Subsection~\ref{sec1.3} for details.
As a key ingredient in the treatment of the Dirichlet boundary condition, we employ the Green function method together with the weighted flux $\mathbf{x}F$, which yields additional boundary information:
\[
 x_NF=-F(\mathbf{x}\cdot\mathbf{n})=0,
 \quad\text{on } \partial\mathbb{R}^N_+=\{x_N=0\}, \quad N=2,3.
\]
This observation is specific to the half-space geometry: the Green function encodes this reflection symmetry and thereby yields the decomposition and estimates for the effective viscous flux $F$ in $\mathbb{R}^N_+ \ (N=2,3)$;
see Subsection~\ref{sec2.2} for details.
Motivated by this structure, it is natural to ask whether an analogous boundary-compatible quantity exists in bounded domains, such as the unit ball, playing the role of $\mathbf{x}F$ in the half-space.
\end{remark}

\subsection{Strategy of the proof}\label{sec1.3}

We now outline the main ideas and difficulties in the proof. The argument relies on the construction of global smooth approximate solutions via a delicate limiting procedure. This approach consists of two fundamental steps. First, we apply the local existence theory for initial densities that are strictly positive (Lemma \ref{l2.1}). Second, we pass to the limit as the lower bound of the initial density tends to zero (see Section \ref{sec4} for details). Bridging these two steps crucially depends on deriving suitable \textit{a priori} estimates that are uniform with respect to this lower bound and ensure the desired regularity.

It should be emphasized that the key techniques developed in \cite{Hoff05, BB23} do not appear to be directly adaptable to the Dirichlet boundary value problem. In addition, without an $L^\infty$ bound of the density, it is unclear how the approach of \cite{P14} could be extended, particularly in the presence of vacuum both in the interior and at spatial infinity. As a consequence, new observations and ideas are required to handle the joint presence of no-slip boundary conditions and far-field vacuum.

To address the issue of far-field vacuum within the $L^p$ framework, we extend the weighted inequality established in \cite{LX19} for nonnegative function $\varrho\in L^1\cap L^\infty$ in $\mathbb{R}^2$ to the half-plane setting under the weaker assumption $L^1\cap L^\zeta$ (see Lemma \ref{l2.4}).
In this spirit, we work in the density integrability class $L^1\cap L^\theta$ for some $\theta\in[\zeta,\infty)$ and analyze the propagation of the $L^\theta$ bound. Actually, from the mass
equation $\eqref{a1}_1$, we derive that
\begin{equation}\label{1.16}
  \frac{\mathrm{d}}{\mathrm{d}t}\int\rho^r\mathrm{d}\mathbf{x}+(r-1)\int\rho^r\divv\mathbf{u}\mathrm{d}\mathbf{x}=0
\end{equation}
for any $1<r<\infty$. This indicates that the effective viscous flux
\begin{equation*}
  F=(2\mu+\lambda)\divv\mathbf{u}-P(\rho)
\end{equation*}
plays a crucial role in generating dissipation for the density at the $L^r$ level.
However, the presence of the Dirichlet boundary necessitates a substantially more delicate analysis of both $F$ and $P(\rho)$, especially under the assumption that $\rho\in L^1\cap L^\theta$.

The first and foremost difficulty lies in the treatment of the effective viscous flux $F$, which therefore demands and indeed receives substantial analytical attention throughout the proof (see Subsection \ref{sec2.2} for details). Unlike the cases of $\mathbb{R}^N$ or bounded domains with slip boundary conditions, the Dirichlet setting yields weaker analytic control from the structural quantity for enhancing regularity. This in turn motivates the introduction of a quantity that formally shares the essential structural advantages of Hoff's framework.
Relying only on the structure of the equations and the boundary conditions, we observe that the specific geometry of the half-space gives rise to an additional boundary
\[
 x_NF=-F(\mathbf{x}\cdot\mathbf{n})=0,
 \quad\text{on } \partial\mathbb{R}^N_+=\{x_N=0\}, \quad N=2,3,
\]
where $\mathbf{n}=-\mathbf{e}_N$ is the unit outward normal vector to $\partial\mathbb{R}^N_+$.
Based on this observation, we exploit the elliptic structures of $\Delta\mathbf{u}$ and $\Delta(\mathbf{x}F)$. In combination with the Green function method, this allows us to decompose
$F$ into two components (see Lemma \ref{l2.6} for details):
\begin{equation*}
F\sim \widetilde{F}+F_P,
\end{equation*}
where (see Lemma \ref{l2.7})
\begin{equation*}
\|\nabla \widetilde{F}\|_{L^q(\mathbb{R}^N_+)}\lesssim \|\rho\dot{\mathbf{u}}\|_{L^q(\mathbb{R}^N_+)}, \quad
\|F_P\|_{L^q(\mathbb{R}^N_+)}\lesssim\|P\|_{L^q(\mathbb{R}^N_+)},
\ \ \ \ \ 1<q<\infty, \ \ N=2,3.
\end{equation*}
Once the $L^q$ bound for the pressure $P$ is established, the analysis effectively reduces to studying $\widetilde{F}$, which serves as a more regular and analytically tractable principal unknown.

The second major difficulty arises in obtaining estimates for $\|\nabla\mathbf{u}\|_{L^q}$ with $2<q<\infty$. In this setting, the classical approach based on the Helmholtz decomposition is no longer applicable, due to the absence of an $L^q$ bound for the vorticity $\curl\mathbf{u}$.
Indeed, although one may isolate a certain component from $F$, the remaining part, when coupled with $\curl\mathbf{u}$, fails to provide any additional regularity. As an alternative strategy, we decompose the velocity field
by analyzing separately the contributions of the source terms in the Lam\'e system (see \eqref{2.19} for $\mathbf{u}=\mathbf{w}_1+\mathbf{w}_2$).
This approach yields a cruder estimate for $\nabla\mathbf{u}$, which forces us to go beyond the framework of $L^4$ estimates and to employ $L^6$ bounds at certain stages of the analysis (see \eqref{3.18} and \eqref{3.41}). In light of this, we adopt the following estimate to the two-dimensional setting (see \eqref{2.20}):
\begin{equation*}
\|\nabla\mathbf{w}_1\|_{L^6}\lesssim \|\nabla^2\mathbf{w}_1\|_{L^{3/2}}
\lesssim\|\rho\dot{\mathbf{u}}\|_{L^{3/2}}\lesssim \|\sqrt{\rho}\|_{L^6}\|\sqrt{\rho}\dot{\mathbf{u}}\|_{L^2}.
\end{equation*}
Although these less tractable estimates do not close on their own, they allow the argument to proceed and ultimately close once the small-energy assumption is fully exploited.

Moreover, obtaining any decay control at infinity within the weak solution framework is particularly delicate and challenging, especially in the presence of Dirichlet boundary conditions.
This is due to the limited usefulness of the available \textit{a priori} bounds,
such as the crude control of $\|\nabla\mathbf{u}\|_{L^q}$, and to the lack of higher-order estimates, which leaves no mechanism to derive decay in this setting.
In view of the tools at hand, the only information we can employ is the $L^\theta$ bound concerning the density. As an attempt, we exploit the material derivative weighted by the density (see \eqref{3.53}), and obtain a growth bound:
\begin{equation*}
\|\nabla\widetilde{F}\|_{L^4}\lesssim \|\rho\dot{\mathbf{u}}\|_{L^4}\lesssim (1+t)^4\cdots.
\end{equation*}
Unlike in the strong-solution setting, the route based on the decay-control for higher derivatives of the velocity field is no longer available, so we revisit the density and the pressure $P(\rho)$. By exploiting the integrability of the pressure together with the spatial embedding in two-dimensional setting, we sharpen the above estimate and reduce the time growth (see \eqref{3.57}):
\begin{equation*}
\|\widetilde{F}\|_{L^\infty} \lesssim
\|\nabla \widetilde{F}\|_{L^4}^{\frac{3}{16}}
\|\widetilde{F}\|_{L^{\frac{52}{3}}}^{\frac{13}{16}}
\lesssim
\|\nabla \widetilde{F}\|_{L^4}^{\frac{3}{16}}
\|\nabla\widetilde{F}\|_{L^{\frac{52}{29}}}^{\frac{13}{16}}
 \lesssim
(1+t)^\frac34
\|\sqrt{\rho}\|_{L^{\frac{52}{3}}}^{\frac{13}{16}}
\|\sqrt{\rho}\dot{\mathbf{u}}\|_{L^2}^{\frac{13}{16}}\cdots,
\end{equation*}
which can now be enclosed under \eqref{1.13}. In fact, the appearance of $\|P\|_{L^6}$ and $\|\sqrt{\rho}\|_{L^{\frac{52}{3}}}$, combined with Lemma \ref{l2.4}, naturally leads to the following choice of the index:
\begin{equation*}
\theta\geq\max\big\{\zeta=4(2\alpha+1)/(\alpha-1), \ 6\gamma, \ 26/3\big\}.
\end{equation*}

Finally, the $L^\theta$ integrability of the density gives rise to additional analytical difficulties in the treatment of the nonlinear terms involving $P(\rho)$, thereby necessitating a more refined analysis. For example, the term
\begin{equation*}
\int_0^T\int\sigma P|\nabla\mathbf{u}|^2\mathrm{d}\mathbf{x}\mathrm{d}t
\end{equation*}
becomes difficult to control analytically for $\sigma=\sigma(t)\triangleq\min\{1,t\}$, particularly in view of the lack of a satisfactory bound for $\|\nabla\mathbf{u}\|_{L^4}$.
To overcome this obstacle, we split the time interval into two regimes and carefully account for the initial-time singularity associated with the velocity gradient by means of an $L^6$ estimate
(see \eqref{3.18} and \eqref{3.19}). This procedure allows us to focus on
\begin{equation*}
\int_{0}^{T}\int\sigma^3P^4\mathrm{d}\mathbf{x}\mathrm{d}t,
\end{equation*}
which can be enclosed by the estimates derived in \eqref{3.33} once the effective viscous flux has been controlled.
This reflects the fact that, in general, the $L^\theta$ framework is less effective than $L^\infty$ estimates in controlling nonlinear terms involving the pressure $P(\rho)$.

Having resolved the above difficulties, we further invoke Zlotnik's inequality (see Lemma \ref{lzlo}) to obtain a time-uniform upper bound for $\|\rho\|_{L^\theta}$. A key step in this procedure is the global control of $\|\widetilde{F}\|_{L^\infty}$, as shown in \eqref{3.54} and \eqref{3.57}, where the density-weighted inequality established in Lemma \ref{l2.4} and the spatial weighted estimate in Lemma \ref{l3.5} play central roles.
With these \textit{a priori} estimates in hand, we obtain the desired regularity for the density,
in particular a level of integrability stronger than that typically available in the Lions--Feireisl's theory. At last, by a standard compactness argument combined with the continuity method, we construct a global weak solution in the sense of Definition \ref{d1.1} that satisfies \eqref{1.14}.

The rest of the paper is organized as follows. In the next section, we recall some known facts and elementary inequalities to be used later, and we introduce the core treatment of the effective viscous flux $F$ in Subsection \ref{sec2.2}.
Section \ref{sec3} is devoted to establishing {\it a priori} estimates. The proof of Theorem \ref{t1.1} is given in Section \ref{sec4}.

\section{Preliminaries}\label{sec2}

In this section we collect some facts and elementary inequalities that will be used later.

\subsection{Auxiliary results and inequalities}

In this subsection we review some known facts and inequalities. First, similar to the proofs in \cite{MN80,LL14}, we have the following result concerning the local existence of strong solutions to the problem \eqref{a1}--\eqref{a4}.
\begin{lemma}\label{l2.1}
Assume that
\begin{equation*}
 \bar{x}^\alpha\rho_0\in L^1,\ \rho_0\in H^{2},\
\mathbf{u}_0\in D^2\cap D_0^1, \ \sqrt{\rho_0}\mathbf{u}_0\in L^2, \
 \inf\limits_{\mathbf{x}\in \widetilde{B}_{R}}\rho_0(\mathbf{x})>0
\end{equation*}
for some large half-ball $\widetilde{B}_{R}$, then there exists a small time $T$ such that the problem \eqref{a1}--\eqref{a4} admits a unique strong solution $(\rho,\mathbf{u})$ satisfying
\begin{equation*}
\bar{x}^\alpha\rho \in L^\infty(0,T; L^{1}),\ \
\rho\in C([0,T]; H^{2}),\
\mathbf{u}\in C([0,T]; D^{2}\cap D^1), \ \inf_{\widetilde{B}_{R}\times[0,T]}\rho(\mathbf{x},t)
\geq \frac12 \inf\limits_{\mathbf{x}\in \widetilde{B}_{R}}\rho_0(\mathbf{x})
>0.
\end{equation*}
\end{lemma}

The following well-known Gagliardo--Nirenberg inequality (see \cite{NI1959}) will be used later.
\begin{lemma}\label{l2.2}
For $p\in [2, \infty)$, $q\in(1, \infty)$, and $r\in (2, \infty)$, there exists some generic constant $C>0$ which may depend on $p$, $q$, and $r$ such that, for $f\in H^1$ and $g\in L^q$ with $\nabla g\in L^r$,
\begin{gather*}
\|f\|_{L^p}\leq C\|f\|_{L^2}^\frac{2}{p}\|\nabla f\|_{L^2}^{1-\frac{2}{p}},\ \ \
\|g\|_{L^\infty}\leq C\|g\|_{L^q}^\frac{q(r-2)}{2r+q(r-2)}\|\nabla g\|_{L^r}^\frac{2r}{2r+q(r-2)}.
\end{gather*}
\end{lemma}

For $\mathbf{v}\in H^{1}_0(\mathbb{R}^2_+)$, define its zero extension $\tilde{\mathbf{v}}$ to $\mathbb{R}^2$. Then $\tilde{\mathbf{v}}\in H^{1}(\mathbb{R}^2)$ and there exists an absolute constant $C>0$ such that
\begin{equation*}
  C^{-1}\|\mathbf{v}\|_{H^{1}(\mathbb{R}^2_+)}\leq \|\tilde{\mathbf{v}}\|_{H^{1}(\mathbb{R}^2)}
  \leq C\|\mathbf{v}\|_{H^{1}(\mathbb{R}^2_+)}.
\end{equation*}
Thus, via this extension argument, the following weighted inequality in
\cite[Theorem B.1]{PL96}, originally stated for $\mathbb{R}^2$, also holds for functions in $H^{1}_0(\mathbb{R}^2_+)$.
\begin{lemma}\label{l2.3}
For $m\in[2,\infty)$ and $\beta\in(1+m/2,\infty)$, there exists a positive constant $C$ such that for all $\mathbf{v}\in H^{1}_0(\mathbb{R}^2_+)$,
\begin{equation*}
\left(\int_{\mathbb{R}^2_+}\frac{|\mathbf{v}|^m}{e+|\mathbf{x}|^2}(\log(e+|\mathbf{x}|^2))^{-\beta}\mathrm{d}\mathbf{x}\right)^{1/m}
\leq C\|\mathbf{v}\|_{L^2(\widetilde{B}_1)}+C\|\nabla\mathbf{v}\|_{L^2(\mathbb{R}^2_+)}.
\end{equation*}
\end{lemma}

Based on Lemma \ref{l2.3} and Poincar\'e's inequality, we obtain the key bound for $\|\varrho\mathbf{v}\|_{L^4}$ with a nonnegative function $\varrho \in L^1\cap L^\zeta$ ($\zeta<\theta$), which extends the inequality in \cite[Lemma 2.4]{LX19} with $\varrho \in L^1(\mathbb{R}^2)\cap L^\infty(\mathbb{R}^2)$.
\begin{lemma}\label{l2.4}
Let $\zeta=4(2\alpha+1)/(\alpha-1)$ with $\alpha\in(1,\infty)$ and $\bar{x}$ be as in \eqref{1.10}. Assume that the nonnegative function $\varrho\in L^1(\mathbb{R}^2_+)\cap L^\zeta(\mathbb{R}^2_+)$ satisfies
\begin{equation*}
M_2\leq \int_{\widetilde{B}_{\eta_*}}\varrho\mathrm{d}\mathbf{x}\leq
\|\varrho\|_{L^1(\mathbb{R}^2_+)}=M_1,\ \ \|\varrho\|_{L^\zeta(\mathbb{R}^2_+)}\leq M_3,\ \
  \bar{x}^\alpha\varrho\in L^1(\mathbb{R}^2_+)
\end{equation*}
for positive constants $M_1, M_2, M_3$, and $\eta_*\geq1$, where $\widetilde{B}_{\eta_*}=B_{\eta_*}\cap\mathbb{R}^2_+$. Then there exists a constant $C>0$ depending only on $M_1$, $M_2$, and $M_3$ such that, for $\mathbf{v}\in D^{1}_0(\mathbb{R}^2_+)$ verifying $\sqrt{\varrho}\mathbf{v}\in L^2(\mathbb{R}^2_+)$,
\begin{align}\label{2.1}
  \|\varrho\mathbf{v}\|_{L^4(\mathbb{R}^2_+)}^4
   \leq
  C\eta_*^8(1+\|\bar{x}^\alpha\varrho\|_{L^{1}(\mathbb{R}^2_+)})
 \big(\|\sqrt{\varrho}\mathbf{v}\|_{L^2(\mathbb{R}^2_+)}^4+\|\nabla\mathbf{v}\|_{L^2(\mathbb{R}^2_+)}^4 \big).
\end{align}
\end{lemma}
\begin{proof}
Define the average of $f$ over the domain $\Omega\subset\mathbb{R}^2_+$ as
\begin{equation*}
  \langle f\rangle_{\Omega}\triangleq \frac{1}{|\Omega|}\int_{\Omega}f(\mathbf{x})\mathrm{d}\mathbf{x}.
\end{equation*}
Using the Poincar\'e inequality in the zero-average form, we deduce from the odd extension argument
that
\begin{equation*}
  \|\mathbf{v}-\langle\mathbf{v}\rangle_{\widetilde{B}_{\eta_*}}\|_{L^2(\widetilde{B}_{\eta_*})}\leq C\eta_*\|\nabla\mathbf{v}\|_{L^2(\widetilde{B}_{\eta_*})},
\end{equation*}
which along with H\"older's inequality yields that
\begin{align*}
  \big|\langle\varrho\rangle_{\widetilde{B}_{\eta_*}}\langle\mathbf{v}\rangle_{\widetilde{B}_{\eta_*}}\big|
  &=\left|\frac{1}{|\widetilde{B}_{\eta_*}|}\int_{\widetilde{B}_{\eta_*}}\big(\langle\varrho\rangle_{\widetilde{B}_{\eta_*}}-\varrho\big)
  \big(\mathbf{v}-\langle\mathbf{v}\rangle_{\widetilde{B}_{\eta_*}}\big)\mathrm{d}\mathbf{x}
  +\frac{1}{|\widetilde{B}_{\eta_*}|}\int_{\widetilde{B}_{\eta_*}}\varrho\mathbf{v}
  \mathrm{d}\mathbf{x}\right|\notag\\
  &\leq C\eta_*^{-2}\big(\|\varrho\|_{L^2(\widetilde{B}_{\eta_*})}+\eta_*^{-1}M_1\big)\|\mathbf{v}-\langle\mathbf{v}\rangle_{\widetilde{B}_{\eta_*}}\|_{L^2(\widetilde{B}_{\eta_*})}
  +C\eta_*^{-2}\|\sqrt{\varrho}\|_{L^2(\widetilde{B}_{\eta_*})}\|\sqrt{\varrho}\mathbf{v}\|_{L^2(\widetilde{B}_{\eta_*})}\notag\\
  &\leq
   C\eta_*^{-1}\big(\|\varrho\|_{L^2(\widetilde{B}_{\eta_*})}+M_1\big)\|\nabla\mathbf{v}\|_{L^2(\widetilde{B}_{\eta_*})}
   +C\eta_*^{-2}M_1^{1/2}\|\sqrt{\varrho}\mathbf{v}\|_{L^2(\widetilde{B}_{\eta_*})}.
\end{align*}

Hence, we have
\begin{equation*}
  \big|\langle\mathbf{v}\rangle_{\widetilde{B}_{\eta_*}}\big|
  \leq
  C(M_1,M_2)\eta_*\big(1+\|\varrho\|_{L^2(\widetilde{B}_{\eta_*})}\big)\|\nabla\mathbf{v}\|_{L^2(\widetilde{B}_{\eta_*})}
   +C(M_1,M_2)\|\sqrt{\varrho}\mathbf{v}\|_{L^2(\widetilde{B}_{\eta_*})},
\end{equation*}
which implies that
\begin{align}\label{2.2}
  \|\mathbf{v}\|_{L^2(\widetilde{B}_{\eta_*})}^2&\leq
  2\int_{\widetilde{B}_{\eta_*}}\big|\mathbf{v}-\langle\mathbf{v}\rangle_{\widetilde{B}_{\eta_*}}\big|^2\mathrm{d}\mathbf{x}
  +2\big|\widetilde{B}_{\eta_*}\big|\big|\langle\mathbf{v}\rangle_{\widetilde{B}_{\eta_*}}\big|^2\notag\\
  &\leq C(M_1,M_2)\eta_*^4\big(1+\|\varrho\|_{L^2(\widetilde{B}_{\eta_*})}^2\big)\|\nabla\mathbf{v}\|_{L^2(\widetilde{B}_{\eta_*})}^2
   +C(M_1,M_2)\eta_*^2\|\sqrt{\varrho}\mathbf{v}\|_{L^2(\widetilde{B}_{\eta_*})}^2.
\end{align}
The local estimate on a half-ball together with Lemma \ref{l2.3} leads to
\begin{align}\label{2.3}
  \|\varrho\mathbf{v}\|_{L^4(\mathbb{R}^2_+)}^4
   &\leq
  \|(\bar{x}^\alpha\varrho)^\frac{2}{\alpha+1}\|_{L^{\frac{\alpha+1}{2}}(\mathbb{R}^2_+)}
  \|\bar{x}^{-\frac{2\alpha}{\alpha+1}}|\mathbf{v}|^4\|_{L^{\frac{2(\alpha+1)}{\alpha-1}}(\mathbb{R}^2_+)}
  \|\varrho^{\frac{\alpha-1}{\alpha+1}}\|_{L^{\frac{4(\alpha+1)(2\alpha+1)}{(\alpha-1)^2}}(\mathbb{R}^2_+)}
  \|\varrho^{3}\|_{L^{\frac{4(2\alpha+1)}{3(\alpha-1)}}(\mathbb{R}^2_+)}\notag\\
  &\leq
  C\eta_*^8(1+\|\bar{x}^\alpha\varrho\|_{L^{1}(\mathbb{R}^2_+)})(1+\|\varrho\|_{L^2(\mathbb{R}^2_+)}^{4})
   \|\varrho\|_{L^{\frac{4(2\alpha+1)}{\alpha-1}}(\mathbb{R}^2_+)}^{3+\frac{\alpha-1}{\alpha+1}}
 \big(\|\sqrt{\varrho}\mathbf{v}\|_{L^2(\mathbb{R}^2_+)}+\|\nabla\mathbf{v}\|_{L^2(\mathbb{R}^2_+)} \big)^4\notag\\
 &\leq
  C(M_1,M_2,M_3)\eta_*^8(1+\|\bar{x}^\alpha\varrho\|_{L^{1}(\mathbb{R}^2_+)})
 \big(\|\sqrt{\varrho}\mathbf{v}\|_{L^2(\mathbb{R}^2_+)}^4+\|\nabla\mathbf{v}\|_{L^2(\mathbb{R}^2_+)}^4 \big),
\end{align}
where we have chosen
\begin{equation*}
  \beta=2\times\frac{4\alpha}{\alpha-1}>1+\frac{4(\alpha+1)}{\alpha-1}.\tag*{\qedhere}
\end{equation*}
\end{proof}

Finally, we introduce a variant of Zlotnik's inequality (see \cite{Zlo}) to derive the uniform (in time) upper bound of $\|\rho\|_{L^\theta}$.

\begin{lemma}\label{lzlo}
Let the function $y$ satisfy\footnote{Although \cite[Lemma 1.3(b)]{Zlo} states the result for $y'(t)=g(y) + b'(t)$, its proof actually works for the inequality as well.}
\begin{equation*}
y'(t)\leq g(y) + b'(t) \ \ \text{on} \ [0, T], \quad y(0) = y^{0},
\end{equation*}
with $g \in C(\mathbb{R})$ and $y, b \in W^{1,1}(0, T)$. If $g(\infty)=-\infty$ and
\begin{equation*}
b(t_{2})-b(t_{1}) \leq N_{0} + N_{1}(t_{2} - t_{1})
\end{equation*}
for all $0 \leq t_{1} < t_{2} \leq T$ with some $N_{0} \geq 0$ and $N_{1} \geq 0$, then
\begin{equation*}
y(t) \leq \max \{y^{0}, \xi^*\} + N_{0} < \infty \ \ \text{on} \ [0, T],
\end{equation*}
where $\xi^*$ is a constant such that
\begin{equation*}
g(\xi) \leq -N_{1} \ \ \text{for} \ \xi \geq \xi^*.
\end{equation*}
\end{lemma}
\subsection{An exposition on the boundary}\label{sec2.2}

In this subsection we explore the boundary condition in more detail and provide partial $L^p$-estimates for the effective viscous flux $F$ and $\nabla\mathbf{u}$.

Using \eqref{1.6} and the vector identity $\Delta \mathbf{u}=\nabla\divv \mathbf{u}-\nabla^\bot\curl \mathbf{u}$, the momentum equation $\eqref{a1}_2$ reads
\begin{equation}\label{2.4}
\rho\dot{\mathbf{u}}=\nabla F-\mu\nabla^\bot\curl\mathbf{u}.
\end{equation}
In particular, if we let
\begin{equation*}
  (u^1,u^2)=(u^1,0), \ \ \text{on} \ \partial\mathbb{R}^2_+,
\end{equation*}
then $\dot{u}^2=0$ on $\partial\mathbb{R}^2_+$, which along with \eqref{2.4} implies that
\begin{align}\label{2.5}
0=\rho\dot{\mathbf{u}}\cdot\mathbf{n}=\mathbf{n}\cdot(\nabla F-\mu\nabla^\bot\curl\mathbf{u}) =\mathbf{n}\cdot\nabla F+\mu\partial_1\partial_2u^1, \ \ \text{on} \ \partial\mathbb{R}^2_+,
\end{align}
where $\mathbf{n}=(0,-1)$ is the unit outward normal vector to $\partial\mathbb{R}^2_+$.
Indeed, the Navier-slip boundary conditions $(u^1,u^2)|_{x_2=0}=(\partial_2u^1,0)|_{x_2=0}$ (see \cite{Hoff05}) supply the boundary information needed to exploit the elliptic structure of $F$.

By contrast, the Dirichlet boundary conditions do not seem to be of use here, and we cannot directly apply the standard elliptic theory for $F$. Fortunately, although no suitable boundary condition exists,
the flatness of the half-space boundary suggests that
\[
 x_NF=-F(\mathbf{x}\cdot\mathbf{n})=0,
 \quad\text{on } \partial\mathbb{R}^N_+=\{x_N=0\}, \quad N=2,3.
\]
This observation directs our attention to  $\Delta\mathbf{u}$ and $\Delta(\mathbf{x}F)$ instead.
Moreover, we need to treat boundary effects separately, for which the Green function method presents itself as a natural choice.

We first recall the Green representation for the Dirichlet problem in smooth domain $\Omega$.
Let $\phi$ solve
\[
\begin{cases}
\Delta\phi=f, & \mathbf{x}\in\Omega,\\
\phi=g, & \mathbf{x}\in\partial\Omega.
\end{cases}
\]
Then, for $\mathbf{x}\in\Omega$,
\[
\phi(\mathbf{x})=\int_{\Omega}G(\mathbf{x},\mathbf{y})f(\mathbf{y})\mathrm{d}\mathbf{y}
+\int_{\partial\Omega}K(\mathbf{x},\boldsymbol{\xi})g(\boldsymbol{\xi})\mathrm{d}S_{\boldsymbol{\xi}}
\triangleq \mathcal{G}_0(f)+\mathcal{G}_b(g),
\]
where $G$ is the Dirichlet Green function,
$K(\mathbf{x},\boldsymbol{\xi})=\partial_{\mathbf n_{\boldsymbol{\xi}}}G(\mathbf{x},\boldsymbol{\xi})$ is the Poisson kernel, and
$\mathcal{G}_b(g)$ denotes the harmonic extension of the boundary datum $g$.
For vector-valued data, the operators $\mathcal{G}_0$ and $\mathcal{G}_b$ are understood componentwise.

More precisely, if the domain is the upper half-space $\mathbb{R}^N_+=\{\mathbf{x}=(\mathbf{x}',x_N)\in\mathbb{R}^N:x_N>0\}$,
$N=2,3$, then the Dirichlet Green function is given by
\[
 G(\mathbf{x},\mathbf{y})=
 \Gamma(\mathbf{x}-\mathbf{y})-\Gamma(\mathbf{x}-\mathbf{y}^*)=
 \begin{cases}
 \displaystyle
 \frac{1}{2\pi}\log\frac{|\mathbf{x}-\mathbf{y}|}{|\mathbf{x}-\mathbf{y}^*|},
 & N=2,\\[2ex]
 \displaystyle
 \frac{1}{4\pi}\left(\frac{1}{|\mathbf{x}-\mathbf{y}^*|}
 -\frac{1}{|\mathbf{x}-\mathbf{y}|}\right),
& N=3,
\end{cases}
\]
where $\mathbf{y}^*=(\mathbf{y}',-y_N)$ is the reflection of $\mathbf{y}=(\mathbf{y}',y_N)$ across
$\partial\mathbb{R}^N_+$, and $\Gamma$ is the fundamental solution satisfying
$\Delta_{\mathbf{x}}\Gamma(\mathbf{x}-\mathbf{y})=\delta_{\mathbf{y}}(\mathbf{x})$.
The associated Poisson kernel admits the unified form
\[
 K(\mathbf{x},\boldsymbol{\xi})=
 \frac{2x_N}{N\omega_N|\mathbf{x}-\boldsymbol{\xi}|^N},
 \qquad
 \mathbf{x}\in\mathbb{R}^N_+,\quad
 \boldsymbol{\xi}\in\partial\mathbb{R}^N_+,\quad N=2,3,
\]
where $\omega_N$ denotes the volume of the unit ball in $\mathbb{R}^N$.

We next use the Green function representation to derive a decomposition of $F$ and the corresponding estimates. The following two lemmas are stated in a form valid for the half-space $\mathbb{R}^N_+ \ (N=2,3)$.

\begin{lemma}[Half-space decomposition of $F$]\label{l2.6}
Let $N\in\{2,3\}$, and let $(\rho,\mathbf{u})$ be a smooth solution
to the corresponding half-space problem \eqref{a1}--\eqref{a4} with the homogeneous
Dirichlet boundary conditions. Then the effective viscous flux $F$ defined in \eqref{1.6}
admits the following decomposition in $\mathbb{R}^N_+$:
\begin{align}\label{2.9}
 F= \widetilde{F}+\frac{2\mu}{3\mu+\lambda}F_P
 =\big(\widetilde{F}_1+\widetilde{F}_2\big)+\frac{2\mu}{3\mu+\lambda}F_P,
\end{align}
where
\begin{equation}\label{2.9*}
\begin{cases}
  \displaystyle
 \widetilde{F}_1\triangleq\frac{\mu+\lambda}{3\mu+\lambda}\big[\divv\mathcal{G}_0\big(\mathbf{x}\divv(\rho\dot{\mathbf{u}})\big)
 -\mathbf{x}\cdot\nabla \mathcal{G}_0\big(\divv(\rho\dot{\mathbf{u}})\big)\big],
 \\[2ex]
 \displaystyle
 \widetilde{F}_2\triangleq\frac{2(2\mu+\lambda)}{3\mu+\lambda}\divv\mathcal{G}_0(\rho\dot{\mathbf{u}})
 -\frac{(N-1)(\mu+\lambda)}{3\mu+\lambda}\mathcal{G}_0\big(\divv(\rho\dot{\mathbf{u}})\big),
  \\[2ex]
  \displaystyle
 \frac{2\mu}{3\mu+\lambda}F_P\triangleq
 \frac{2\mu}{3\mu+\lambda}\big[\divv\mathcal{G}_0(\nabla P)-P\big].
\end{cases}
\end{equation}
\end{lemma}
\begin{proof}
Taking the divergence of the momentum equation $\eqref{a1}_2$ and using the definition of the effective viscous flux $F$ in \eqref{1.6},
we obtain
\begin{equation}\label{2.10}
 \Delta F=
 \divv(\rho\dot{\mathbf{u}}).
\end{equation}
Hence the Green representation yields
 \begin{equation}\label{2.11}
 F=\mathcal{G}_0\big(\divv(\rho\dot{\mathbf{u}})\big)
 +\mathcal{G}_b\big(F|_{\partial\mathbb{R}^N_+}\big).
\end{equation}

From $\eqref{a1}_2$, we obtain the vector-valued Poisson system for $N=2,3$:
\begin{equation}\label{2.6}
\begin{cases}
\displaystyle
 \mu\Delta\mathbf{u}=\rho\dot{\mathbf{u}}
 +\frac{\mu}{2\mu+\lambda}\nabla P
 -\frac{\mu+\lambda}{2\mu+\lambda}\nabla F, & \mathbf{x}\in\mathbb{R}^N_+,\\[1ex]
 \mathbf{u}=\mathbf 0, & \mathbf{x}\in\partial\mathbb{R}^N_+.
\end{cases}
\end{equation}
Set
\[
 \beta=\frac{\mu+\lambda}{2(2\mu+\lambda)}.
\]
Noting that
\[
 \Delta(\mathbf{x}F)=\mathbf{x}\Delta F+2\nabla F,
\]
we infer from \eqref{2.6} and \eqref{2.10} that
\begin{align}\label{2.12}
 \Delta\big(\beta\mathbf{x}F+\mu\mathbf{u}\big)
 &=\beta\mathbf{x}\Delta F
 +\frac{\mu+\lambda}{2\mu+\lambda}\nabla F+\mu\Delta\mathbf{u}
 \notag\\
 &=\beta\mathbf{x}\divv(\rho\dot{\mathbf{u}})
 +\frac{\mu}{2\mu+\lambda}\nabla P
 +\rho\dot{\mathbf{u}}\triangleq \mathbf{U}.
\end{align}
Moreover, since $\mathbf{u}=\mathbf 0$ on $\partial\mathbb{R}^N_+$, we have
\begin{equation}\label{2.7}
 \beta \mathbf{x}F+\mu \mathbf{u}
 =\mathcal{G}_0(\mathbf{U})+\beta\mathcal{G}_b\big((\boldsymbol{\xi}F)|_{\partial\mathbb{R}^N_+}\big).
\end{equation}
Taking the divergence of \eqref{2.7}, we obtain
\begin{align*}
 \divv\mathcal{G}_0(\mathbf{U})+\beta\divv\mathcal{G}_b\big((\boldsymbol{\xi}F)|_{\partial\mathbb{R}^N_+}\big)
 &=\beta\divv(\mathbf{x}F)+\mu\divv\mathbf{u}
 \\
 &=\beta\big(\mathbf{x}\cdot\nabla F+NF\big)
 +\frac{\mu}{2\mu+\lambda}(F+P)
 \\
 &=\frac{\mu+\lambda}{2(2\mu+\lambda)}\mathbf{x}\cdot\nabla
 \Big[\mathcal{G}_0\big(\divv(\rho\dot{\mathbf{u}})\big)
 +\mathcal{G}_b\big(F|_{\partial\mathbb{R}^N_+}\big)\Big]
 \notag\\
  &\quad
 +\frac{N(\mu+\lambda)+2\mu}{2(2\mu+\lambda)}F
 +\frac{\mu}{2\mu+\lambda}P,
\end{align*}
and hence
\begin{align}\label{2.13}
 \frac{N(\mu+\lambda)+2\mu}{2(2\mu+\lambda)}F&=\divv\mathcal{G}_0(\mathbf{U})-\frac{\mu+\lambda}{2(2\mu+\lambda)}\mathbf{x}\cdot\nabla
 \mathcal{G}_0\big(\divv(\rho\dot{\mathbf{u}})\big)
 -\frac{\mu}{2\mu+\lambda} P
 \notag\\
 &\quad
 +\frac{\mu+\lambda}{2(2\mu+\lambda)}\Big[\divv\mathcal{G}_b\big((\boldsymbol{\xi}F)|_{\partial\mathbb{R}^N_+}\big)
  -\mathbf{x}\cdot\nabla\mathcal{G}_b\big(F|_{\partial\mathbb{R}^N_+}\big)\Big].
\end{align}

We now analyze the last boundary contribution in \eqref{2.13}.
For $\boldsymbol{\xi}\in\partial\mathbb{R}^N_+$, we have $\xi_N=0$ and hence
$x_N=(\mathbf{x}-\boldsymbol{\xi})\cdot\mathbf{e}_N$.
Thus, the Poisson kernel in the half-space
\[
 K(\mathbf{x},\boldsymbol{\xi})=
 \frac{2}{N\omega_N}\frac{(\mathbf{x}-\boldsymbol{\xi})\cdot\mathbf{e}_N}
 {|\mathbf{x}-\boldsymbol{\xi}|^N}
\]
is homogeneous of degree $1-N$ in $\mathbf{x}-\boldsymbol{\xi}$.
Euler's identity for homogeneous functions gives
\begin{equation}\label{2.8}
 -(\mathbf{x}-\boldsymbol{\xi})\cdot\nabla_{\mathbf{x}}K(\mathbf{x},\boldsymbol{\xi})
 =(N-1)K(\mathbf{x},\boldsymbol{\xi}),
 \qquad\mathbf{x}\in\mathbb{R}^N_+,\quad
 \boldsymbol{\xi}\in\partial\mathbb{R}^N_+,\quad N=2,3.
\end{equation}
Therefore,
\begin{align}\label{2.8*}
  \divv\mathcal{G}_b\big((\boldsymbol{\xi}F)|_{\partial\mathbb{R}^N_+}\big)
  -\mathbf{x}\cdot\nabla\mathcal{G}_b\big(F|_{\partial\mathbb{R}^N_+}\big)
  &=
  \int_{\partial\mathbb{R}^N_+}\Big[\partial _{x_j}\big(K(\mathbf{x},\boldsymbol{\xi})\xi^jF(\boldsymbol{\xi})\big)
  -\mathbf{x}\cdot\nabla_{\mathbf{x}}\big(K(\mathbf{x},\boldsymbol{\xi})F(\boldsymbol{\xi})\big)\Big]\mathrm{d}S_{\boldsymbol{\xi}} \notag\\
  &=
  -\int_{\partial\mathbb{R}^N_+}(\mathbf{x}-\boldsymbol{\xi})\cdot\nabla_{\mathbf{x}}K(\mathbf{x},\boldsymbol{\xi})
  F(\boldsymbol{\xi})\mathrm{d}S_{\boldsymbol{\xi}} \notag\\
  &=
  (N-1)\int_{\partial\mathbb{R}^N_+}K(\mathbf{x},\boldsymbol{\xi})F(\boldsymbol{\xi})\mathrm{d}S_{\boldsymbol{\xi}}
  \notag\\
  &=(N-1)\mathcal{G}_b\big(F|_{\partial\mathbb{R}^N_+}\big)
  \notag\\
  &=(N-1)\big[F-\mathcal{G}_0\big(\divv(\rho\dot{\mathbf{u}})\big)\big],
\end{align}
where we have used \eqref{2.11} in the last equality.

Combining \eqref{2.8*} with \eqref{2.13}, and then substituting the expression for
$\mathbf{U}$ given in \eqref{2.12} into the resulting identity, we deduce that
\begin{align*}
 \frac{N(\mu+\lambda)+2\mu}{2(2\mu+\lambda)}F&=\divv\mathcal{G}_0(\mathbf{U})-\frac{\mu+\lambda}{2(2\mu+\lambda)}\mathbf{x}\cdot\nabla
 \mathcal{G}_0\big(\divv(\rho\dot{\mathbf{u}})\big)
 -\frac{\mu}{2\mu+\lambda}P
 \notag\\
 &\quad+\frac{(N-1)(\mu+\lambda)}{2(2\mu+\lambda)}
 \big[F-\mathcal{G}_0\big(\divv(\rho\dot{\mathbf{u}})\big)\big]
 \notag\\
 &=
 \frac{\mu+\lambda}{2(2\mu+\lambda)}\big[\divv\mathcal{G}_0\big(\mathbf{x}\divv(\rho\dot{\mathbf{u}})\big)
 -\mathbf{x}\cdot\nabla \mathcal{G}_0\big(\divv(\rho\dot{\mathbf{u}})\big)\big]
 +\divv\mathcal{G}_0(\rho\dot{\mathbf{u}})
 \notag\\
 &\quad
 -\frac{(N-1)(\mu+\lambda)}{2(2\mu+\lambda)}\mathcal{G}_0\big(\divv(\rho\dot{\mathbf{u}})\big)
 +\frac{\mu}{2\mu+\lambda}\big[\divv\mathcal{G}_0(\nabla P)-P\big]
+\frac{(N-1)(\mu+\lambda)}{2(2\mu+\lambda)}F.
\end{align*}
Collecting the terms involving $F$, we arrive at the desired decomposition.
\end{proof}

The boundary-effect-free decomposition allows us to estimate the Green potentials term by term.

\begin{lemma}[Half-space estimates for $F$]\label{l2.7}
Let $N\in\{2,3\}$. Then, for any $1<p<\infty$, there exist positive constants
$C=C(p,N,\mu,\lambda)$ and $\Lambda(p,N)$ such that
\begin{align}\label{2.16}
 \|\nabla \widetilde{F}\|_{L^p(\mathbb{R}^N_+)}\leq C\|\rho\dot{\mathbf{u}}\|_{L^p(\mathbb{R}^N_+)},\qquad
 \|F_P\|_{L^p(\mathbb{R}^N_+)}\leq \Lambda(p,N)\|P\|_{L^p(\mathbb{R}^N_+)}.
\end{align}
In addition, when $1<p<N$, we have
\begin{equation}\label{ker}
  \| \widetilde{F}\|_{L^\frac{Np}{N-p}(\mathbb{R}^N_+)}
  \leq C\|\nabla \widetilde{F}\|_{L^p(\mathbb{R}^N_+)}\leq
  C\|\rho\dot{\mathbf{u}}\|_{L^p(\mathbb{R}^N_+)}.
\end{equation}
\end{lemma}
\begin{proof}

Set
\[
 \mathbf{V}\triangleq\rho\dot{\mathbf{u}}.
\]
By the standard kernel bounds for the half-space Dirichlet Green function, the order-zero operators
\[
 \mathbf{V}\mapsto \nabla\divv\mathcal{G}_0(\mathbf{V}), \qquad
 \mathbf{V}\mapsto \nabla\mathcal{G}_0(\divv\mathbf{V}), \qquad
 P\mapsto \divv\mathcal{G}_0(\nabla P)-P
\]
are bounded on $L^p(\mathbb{R}^N_+)$ for any $1<p<\infty$. Consequently,
\begin{align}\label{M1}
  \|F_P\|_{L^p(\mathbb{R}^N_+)}&=\|\divv\mathcal{G}_0(\nabla P)-P\|_{L^p(\mathbb{R}^N_+)}
 \leq \Lambda(p,N)\|P\|_{L^p(\mathbb{R}^N_+)},
 \\[1ex]
 \|\nabla \widetilde{F}_2\|_{L^p(\mathbb{R}^N_+)}&\leq
 C(N,\mu,\lambda)\big(\|\nabla\divv\mathcal{G}_0(\mathbf{V})\|_{L^p(\mathbb{R}^N_+)}
 +\|\nabla\mathcal{G}_0(\divv\mathbf{V})\|_{L^p(\mathbb{R}^N_+)}\big)\notag\\
 &\leq C(p,N,\mu,\lambda)\|\mathbf{V}\|_{L^p(\mathbb{R}^N_+)}. \label{M2}
\end{align}
It remains to handle the commutator term in $\widetilde{F}_1$. Set
\[
 \mathscr{T}\mathbf{V}\triangleq
 \divv\mathcal{G}_0(\mathbf{x}\divv\mathbf{V})-\mathbf{x}\cdot\nabla\mathcal{G}_0(\divv\mathbf{V}).
\]
By the Green representation,
\[
 \mathscr{T}\mathbf{V}(\mathbf{x})=\int_{\mathbb{R}^N_+}
 H(\mathbf{x},\mathbf{y})\divv_{\mathbf{y}}\mathbf{V}(\mathbf{y})\mathrm{d}\mathbf{y},
 \qquad
 H(\mathbf{x},\mathbf{y})=(\mathbf{y}-\mathbf{x})\cdot\nabla_{\mathbf{x}}G(\mathbf{x},\mathbf{y}).
\]
Since $H(\mathbf{x},\mathbf{y})$ vanishes for $\mathbf{y}\in\partial\mathbb{R}^N_+$,
integration by parts shows that $\nabla\mathscr{T}$ is represented by the kernel
$-\nabla_{\mathbf{x}}\nabla_{\mathbf{y}}H(\mathbf{x},\mathbf{y})$,
which has a reflected boundary singularity.
By the $L^p$ boundedness of the corresponding reflected
Calder\'on--Zygmund singular integral operator in the half-space, we obtain
\begin{equation}\label{M3}
 \|\nabla \widetilde{F}_1\|_{L^p(\mathbb{R}^N_+)}\leq
 C(\mu,\lambda)\|\nabla\mathscr{T}\mathbf{V}\|_{L^p(\mathbb{R}^N_+)}
 \leq C(p,N,\mu,\lambda)\|\mathbf{V}\|_{L^p(\mathbb{R}^N_+)},
 \quad 1<p<\infty.
\end{equation}

On the other hand, the kernel bounds
\[
 |\nabla_{\mathbf{y}}H(\mathbf{x},\mathbf{y})|+|\nabla_{\mathbf{x}}G(\mathbf{x},\mathbf{y})|
 +|\nabla_{\mathbf{y}}G(\mathbf{x},\mathbf{y})|\leq
 C\big(|\mathbf{x}-\mathbf{y}|^{1-N}+|\mathbf{x}-\mathbf{y}^*|^{1-N}\big)
\]
imply the lower-order estimate. Indeed, since $\mathbf{V}|_{\partial\mathbb{R}^N_+}=0$,
extending $\mathbf{V}$ by zero to $\mathbb{R}^N$ and applying the
Hardy--Littlewood--Sobolev inequality in $\mathbb{R}^N$ yield
\begin{equation}\label{2.24}
 \|\widetilde{F}_1\|_{L^{\frac{Np}{N-p}}(\mathbb{R}^N_+)}+\|\widetilde{F}_2\|_{L^{\frac{Np}{N-p}}(\mathbb{R}^N_+)}
 \leq C(p,N,\mu,\lambda)\|\mathbf{V}\|_{L^p(\mathbb{R}^N_+)}, \quad 1<p<N.
\end{equation}
Substituting the expressions for $\mathbf{V}$ gives the desired estimates, completing the proof.
\end{proof}

In addition, the bound of $\|\nabla\mathbf{u}\|_{L^q}$ is obtained via elliptic estimates for the constituent velocity fields, which arise from decomposing the source term of the linear Lam\'e system.
\begin{lemma}\label{l2.8}
Let $(\rho,\mathbf{u})$ be a smooth solution to the problem \eqref{a1}--\eqref{a4}.
For any $2<q<\infty$, there exists a constant $C>0$ depending only on $\mu$ and $q$ such that
\begin{equation}\label{2.18}
  \|\nabla\mathbf{u}\|_{L^q}\leq C\|\sqrt{\rho}\|_{L^q}\|\sqrt{\rho}\dot{\mathbf{u}}\|_{L^2}+C\|P\|_{L^q}.
\end{equation}
\end{lemma}
\begin{proof}
Let $\mathbf{w}_1$ and $\mathbf{w}_2$ satisfy
\begin{align}\label{2.19}
\begin{array}{ll}
\begin{cases}
\mu\Delta\mathbf{w}_1+(\mu+\lambda)\nabla\divv\mathbf{w}_1=\rho\dot{\mathbf{u}}, & \mathbf{x}\in\mathbb{R}^2_+,\\
\mathbf{w}_1=0,& \mathbf{x}\in\partial\mathbb{R}^2_+,
\end{cases}\ \ \ \ \
\begin{cases}
\mu\Delta\mathbf{w}_2+(\mu+\lambda)\nabla\divv\mathbf{w}_2=\nabla P, & \mathbf{x}\in\mathbb{R}^2_+,\\
\mathbf{w}_2=0,& \mathbf{x}\in\partial\mathbb{R}^2_+.
\end{cases}
\end{array}
\end{align}
Accordingly, we set $\mathbf{u}=\mathbf{w}_1+\mathbf{w}_2$.
By the elliptic estimates for the Lam\'e system, together with the Sobolev embedding in two dimensions, we obtain
\begin{align}\label{2.20}
 \|\nabla\mathbf{u}\|_{L^q}
 &\leq
 \|\nabla\mathbf{w}_1\|_{L^q}+\|\nabla\mathbf{w}_2\|_{L^q} \notag\\
 &\leq
 C\|\nabla^2\mathbf{w}_1\|_{L^{\frac{2q}{q+2}}}+\|\nabla\mathbf{w}_2\|_{L^q} \notag\\
 &\leq
 C\Big(\frac{1}{\mu}+\frac{1}{2\mu+\lambda}\Big)\Big(\|\rho\dot{\mathbf{u}}\|_{L^{\frac{2q}{q+2}}}
 +\|P\|_{L^q}\Big) \\
 &\leq
 C\big(\|\sqrt{\rho}\|_{L^q}\|\sqrt{\rho}\dot{\mathbf{u}}\|_{L^2}+\|P\|_{L^q}\big),\notag
\end{align}
where we have used $\mu+\lambda\geq 0$ and H\"older's inequality in the last step. The proof is complete.
\end{proof}

\section{\textit{A priori} estimates}\label{sec3}

In this section we derive necessary {\it a priori} bounds for the strong solution to \eqref{a1}--\eqref{a4}, whose existence is guaranteed by Lemma \ref{l2.1}.
These bounds are independent of the lower bound of $\rho$, the initial regularity, and the time of existence. More precisely, fix $T>0$ and let $(\rho,\mathbf{u})$ be the corresponding strong solution on $\mathbb{R}^2_+\times(0,T]$.

First of all, if we multiply \eqref{a1}$_1$ by a cut-off function and use a standard limit procedure, then one can obtain that
\begin{align}\label{wz}
\|\rho(t)\|_{L^1}=\|\rho_0\|_{L^1},\ \ 0\leq t\leq T.
\end{align}
Next, for subsequent estimates, we define auxiliary functionals as
\begin{align}\label{3.1}
  &A_1(T)\triangleq \sup_{t\in[0,T]}\big(\sigma\|\nabla \mathbf{u}\|_{L^2}^2\big)+\int_0^T\sigma\|\sqrt{\rho} \dot{\mathbf{u}}\|_{L^2}^2\mathrm{d}t,\\
   &A_2(T)\triangleq \sup_{t\in[0,T]}\big(\sigma^3\|\sqrt{\rho} \dot{\mathbf{u}}\|_{L^2}^2\big)
   +\int_0^T\sigma^3\|\nabla \dot{\mathbf{u}}\|_{L^2}^2\mathrm{d}t,\label{3.2}\\
   &A_3(T)\triangleq \sup_{t\in[0,T]}\|\nabla \mathbf{u}\|_{L^2}^2
   +\int_0^{T}\|\sqrt{\rho}\dot{\mathbf{u}}\|_{L^2}^2\mathrm{d}t,
   \label{3.3}
\end{align}
where $\sigma=\sigma(t)\triangleq\min\{1,t\}$. Then we derive the following key {\it a priori} estimates on $(\rho, \mathbf{u})$.

\begin{proposition}\label{p3.1}
Under the conditions of Theorem $\ref{t1.1}$, there exist two positive constants $K\geq M^2$ and $\varepsilon$ depending only on $\alpha, \hat{\rho}, M, a, \gamma, \mu, \lambda$, and $\eta_0$ such that if $(\rho, \mathbf{u})$ is a strong solution to the problem \eqref{a1}--\eqref{a4} satisfying
\begin{align}\label{3.4}
\sup_{\mathbb{R}^2_+\times[0,T]}\|\rho\|_{L^{\theta}}\le2\hat{\rho},\ \
A_1(T)+A_2(T)\leq 2C_0^\frac{\alpha+2}{7\alpha+5}, \ \
A_3(\sigma(T))\leq 3K,
\end{align}
then the following improved bounds hold
\begin{align}\label{3.5}
\sup_{\mathbb{R}^2_+\times[0,T]}\|\rho\|_{L^{\theta}}\le\frac{7}{4}\hat{\rho},\ \
A_1(T)+A_2(T)\leq C_0^\frac{\alpha+2}{7\alpha+5},\ \
A_3(\sigma(T))\leq 2K,
\end{align}
provided $C_0\leq \varepsilon$.
\end{proposition}

Before proving Proposition \ref{p3.1}, we establish some necessary \textit{a priori} estimates, see Lemmas \ref{l3.1}--\ref{l3.6} below. We begin with the basic energy estimate of $(\rho, \mathbf{u})$.

\begin{lemma}\label{l3.1}
It holds that
\begin{align}\label{3.6}
\sup_{0\leq t\leq T}\int\bigg(\frac{1}{2}\rho |\mathbf{u}|^2+\frac{P}{\gamma-1}\bigg)\mathrm{d}\mathbf{x}+\int_0^T\big[\mu\|\nabla \mathbf{u}\|_{L^2}^2+(\mu+\lambda)\|\divv \mathbf{u}\|_{L^2}^2\big]\mathrm{d}t
\leq C_0.
\end{align}
\end{lemma}
\begin{proof}
From the mass equation $\eqref{a1}_1$, one has that
\begin{equation}\label{3.7}
\frac{P_t}{\gamma-1}+\frac{\divv(P\mathbf{u})}{\gamma-1}+P\divv \mathbf{u}=0.
\end{equation}
Multiplying $\eqref{a1}_2$ by $\mathbf{u}$ and integrating by parts, and then adding the resultant to \eqref{3.7}, we infer from \eqref{a3} and \eqref{a4} that
\begin{align*}
\frac{\mathrm{d}}{\mathrm{d}t}\int\bigg(\frac{1}{2}\rho |\mathbf{u}|^2+\frac{P}{\gamma-1}\bigg)\mathrm{d}\mathbf{x}+\int\big[\mu|\nabla\mathbf{u}|^2
+(\mu+\lambda)(\divv\mathbf{u})^2\big]\mathrm{d}\mathbf{x}
=0,
\end{align*}
which yields \eqref{3.6} after integrating the above equality with respect to $t$ over $(0,T)$.
\end{proof}

The next result concerns the preliminary bounds of $A_1(T)$ and $A_2(T)$.
\begin{lemma}\label{l3.2}
Let $(\rho,\mathbf{u})$ be a smooth solution of \eqref{a1}--\eqref{a4} on $\mathbb{R}^2_+\times(0,T]$ satisfying $\|\rho\|_{L^{\theta}}\le2\hat{\rho}$. Then there exist positive constants $C$ and $\varepsilon_1$ depending only on $\alpha, \hat{\rho}, a, \gamma, \mu$, and $\lambda$ such that
\begin{gather}\label{3.8}
A_1(T)\leq  CC_0^\frac{2(\alpha+2)}{7\alpha+5}+C\int_0^T \sigma\|\nabla\mathbf{u}\|_{L^3}^3\mathrm{d}t
 +C\int_{0}^T\sigma^3\|P\|_{L^4}^4\mathrm{d}t,
\\
A_2(T)\leq CA_1(T)+C
 \Big(A_1(T)+A_1^\frac32(T)\Big)A_2^\frac12(T)+C\int_0^T\sigma^3\|P\|_{L^4}^4\mathrm{d}t \label{3.9}
\end{gather}
provided $C_0\leq \varepsilon_1$.
\end{lemma}
\begin{proof}
Using \eqref{3.6} and $P=a\rho^\gamma$, we obtain the following bound of $\|P\|_{L^q}$ with $q\in[1,20]$:
\begin{equation}\label{3.10}
 \|P\|_{L^q}=a\|\rho\|_{L^{q\gamma}}^\gamma\leq  a\|\rho\|_{L^\gamma}^\frac{4\gamma(2\alpha+1)-q\gamma(\alpha-1)}{q(7\alpha+5)}
 \|\rho\|_{L^{\theta}}^\frac{4\gamma(q-1)(2\alpha+1)}{q(7\alpha+5)}
 \leq C(a,\gamma,\hat{\rho})C_0^\frac{4(2\alpha+1)-q(\alpha-1)}{q(7\alpha+5)}, \ \ 0\leq t\leq T.
\end{equation}
From $\eqref{a1}_1$ and \eqref{1.6}, the momentum equation $\eqref{a1}_2$ reads
\begin{equation}\label{3.11}
\rho\dot{\mathbf{u}}+\nabla P=\mu\Delta\mathbf{u}+(\mu+\lambda)\nabla\divv \mathbf{u}.
\end{equation}

\textbf{Estimate for $A_1(T)$}.
Multiplying \eqref{3.11} by $\sigma\dot{\mathbf{u}}$ and integrating the resultant over $\mathbb{R}^2_+$, we get that
\begin{align}\label{3.12}
\int\sigma\rho |\dot{\mathbf{u}}|^2\mathrm{d}\mathbf{x}=
\int\sigma[-\dot{\mathbf{u}}\cdot\nabla P+\mu\Delta\mathbf{u}\cdot\dot{\mathbf{u}}
+(\mu+\lambda)\dot{\mathbf{u}}\cdot\nabla\divv \mathbf{u}]\mathrm{d}\mathbf{x}.
\end{align}
For the first term on the right-hand side of \eqref{3.12}, it follows from $\eqref{a1}_1$ and $P=a\rho^\gamma$ that
\begin{align}\label{3.13}
-\int\sigma\dot{\mathbf{u}}\cdot\nabla P\mathrm{d}\mathbf{x}
&=\int \sigma P\divv\mathbf{u}_t\mathrm{d}\mathbf{x}
-\int\sigma\mathbf{u}\cdot\nabla\mathbf{u}\cdot\nabla P\mathrm{d}\mathbf{x}
\notag\\
&=\frac{\mathrm{d}}{\mathrm{d}t}\int\sigma P\divv\mathbf{u} \mathrm{d}\mathbf{x}
-\int\sigma\divv\mathbf{u}P'(\rho)\rho_t\mathrm{d}\mathbf{x}
-\int\sigma' P\divv\mathbf{u} \mathrm{d}\mathbf{x}
-\int\sigma\mathbf{u}\cdot\nabla\mathbf{u}\cdot\nabla P\mathrm{d}\mathbf{x}\notag\\
&=\frac{\mathrm{d}}{\mathrm{d}t}\int \sigma P\divv\mathbf{u}\mathrm{d}\mathbf{x}+
\int\sigma(\divv\mathbf{u})^2P'(\rho)\rho\mathrm{d}\mathbf{x}+
\int\sigma\mathbf{u}\cdot\nabla P\divv\mathbf{u}\mathrm{d}\mathbf{x}
-\int\sigma' P\divv\mathbf{u} \mathrm{d}\mathbf{x}\notag\\
&\quad
-\int\sigma\mathbf{u}\cdot\nabla\mathbf{u}\cdot\nabla P\mathrm{d}\mathbf{x}\notag\\
&=\frac{\mathrm{d}}{\mathrm{d}t}\int\sigma P\divv\mathbf{u}\mathrm{d}\mathbf{x}+
\int\sigma(\divv\mathbf{u})^2(P'(\rho)\rho-P)\mathrm{d}\mathbf{x}
+\int\sigma P\partial_iu^j\partial_ju^i\mathrm{d}\mathbf{x}
-\int\sigma' P\divv\mathbf{u} \mathrm{d}\mathbf{x}\notag\\
&\leq\frac{\mathrm{d}}{\mathrm{d}t}\int\sigma P\divv\mathbf{u}\mathrm{d}\mathbf{x}+
C\int\sigma P|\nabla\mathbf{u}|^2\mathrm{d}\mathbf{x}
+C\int\sigma' P|\nabla\mathbf{u}|\mathrm{d}\mathbf{x}.
\end{align}
Integrating the remaining terms in \eqref{3.12} by parts, we obtain that
\begin{align}\label{3.14}
&\mu\int\sigma\Delta\mathbf{u}\cdot\dot{\mathbf{u}}\mathrm{d}\mathbf{x}
+(\mu+\lambda)\int\sigma\dot{\mathbf{u}}\cdot\nabla\divv \mathbf{u}\mathrm{d}\mathbf{x}\notag\\
&=
-\frac{\mu}{2}\frac{\mathrm{d}}{\mathrm{d}t}\big(\sigma\|\nabla\mathbf{u}\|^2_{L^2}\big)
+\frac{\mu}{2}\sigma'\|\nabla\mathbf{u}\|^2_{L^2}
-\mu\int\sigma\partial_iu^j\partial_i(u^k\partial_ku^j)\mathrm{d}\mathbf{x}
-\frac{\mu+\lambda}{2}\frac{\mathrm{d}}{\mathrm{d}t}\big(\sigma\|\divv\mathbf{u}\|^2_{L^2}\big)\notag\\
&\quad+\frac{\mu+\lambda}{2}\sigma'\|\divv\mathbf{u}\|^2_{L^2}
-(\mu+\lambda)\int\sigma\divv \mathbf{u}\divv (\mathbf{u}\cdot\nabla\mathbf{u})\mathrm{d}\mathbf{x}\notag\\
&\leq -\frac{1}{2}\frac{\mathrm{d}}{\mathrm{d}t}\big[\mu\sigma\|\nabla\mathbf{u}\|^2_{L^2}+(\mu+\lambda)\sigma\|\divv\mathbf{u}\|^2_{L^2}\big]
+C\|\nabla\mathbf{u}\|^2_{L^2}+C\int \sigma|\nabla\mathbf{u}|^3\mathrm{d}\mathbf{x},
\end{align}
owing to $0\leq \sigma,\sigma'\leq 1$.
As a result, substituting \eqref{3.13} and \eqref{3.14} into \eqref{3.12} leads to
\begin{align}\label{3.15}
\frac{\mathrm{d}}{\mathrm{d}t}\big(\sigma B(t)\big)+\int\sigma\rho |\dot{\mathbf{u}}|^2\mathrm{d}\mathbf{x}
&\leq C\|\nabla\mathbf{u}\|^2_{L^2}+C\int \sigma|\nabla\mathbf{u}|^3\mathrm{d}\mathbf{x}
+C\int\sigma' P|\nabla\mathbf{u}|\mathrm{d}\mathbf{x}
+C\int\sigma P|\nabla\mathbf{u}|^2\mathrm{d}\mathbf{x},
\end{align}
where
\begin{align}\label{3.16}
  B(t)&\triangleq \frac{\mu}{2}\|\nabla\mathbf{u}\|^2_{L^2}+\frac{\mu+\lambda}{2}\|\divv\mathbf{u}\|^2_{L^2}
  -\int P\divv\mathbf{u}\mathrm{d}\mathbf{x}\notag\\
  &\geq\frac{\mu}{2}\|\nabla\mathbf{u}\|^2_{L^2}+\frac{\mu+\lambda}{2}\|\divv\mathbf{u}\|^2_{L^2}
  -\|P\|_{L^2}\|\divv\mathbf{u}\|_{L^2}\notag\\
  &\geq\frac{\mu}{2}\|\nabla\mathbf{u}\|^2_{L^2}+\frac{\mu+\lambda}{4}\|\divv\mathbf{u}\|^2_{L^2}
  -C(\hat{\rho})C_0^\frac{6(\alpha+1)}{7\alpha+5},
\end{align}
due to \eqref{3.10}.
Integrating \eqref{3.15} with respect to $t$ over $(0,T)$, one infers from \eqref{3.6} that
\begin{align}\label{3.17}
 &\sup_{t\in[0,T]}\big(\sigma B(t)\big)+\int_0^T\int\sigma\rho |\dot{\mathbf{u}}|^2\mathrm{d}\mathbf{x}\mathrm{d}t
 \notag\\&\leq C(\hat{\rho})C_0^\frac{6(\alpha+1)}{7\alpha+5}+
C\int_0^T\int \sigma|\nabla\mathbf{u}|^3\mathrm{d}\mathbf{x}\mathrm{d}t
+C\int_0^T\int\sigma' P|\nabla\mathbf{u}|\mathrm{d}\mathbf{x}\mathrm{d}t
 +C\int_0^T\int\sigma P|\nabla\mathbf{u}|^2\mathrm{d}\mathbf{x}\mathrm{d}t\notag\\
 &\leq C(\hat{\rho})C_0^\frac{3(\alpha+1)}{7\alpha+5}+C\int_0^T\int \sigma|\nabla\mathbf{u}|^3\mathrm{d}\mathbf{x}\mathrm{d}t
 +C\int_0^T\int\sigma P|\nabla\mathbf{u}|^2\mathrm{d}\mathbf{x}\mathrm{d}t
\end{align}
provided $C_0\leq 1$, where in the last inequality we have used the following estimate
\begin{equation*}
  \int_0^T\int\sigma' P|\nabla\mathbf{u}|\mathrm{d}\mathbf{x}\mathrm{d}t\leq
   \int_0^{\sigma(T)}\|P\|_{L^2}\|\nabla\mathbf{u}\|_{L^2}\mathrm{d}t
   \leq C(\hat{\rho})C_0^\frac{3(\alpha+1)}{7\alpha+5}\int_0^{\sigma(T)}\big(\|\nabla\mathbf{u}\|_{L^2}^2\big)^\frac12\mathrm{d}t
   \leq C(\hat{\rho})C_0^\frac{3(\alpha+1)}{7\alpha+5}.
\end{equation*}

To handle the last term in \eqref{3.17}, we decompose the time axis into two parts: a short-time interval and a long one. If $C_0\leq1$, it follows from \eqref{3.10}, Lemma \ref{l2.8}, and H\"older's inequality that
\begin{align}\label{3.18}
  \int_0^{\sigma(T)}\int\sigma P|\nabla\mathbf{u}|^2\mathrm{d}\mathbf{x}\mathrm{d}t
  &\leq  \int_0^{\sigma(T)}\sigma\|P\|_{L^3}\|\nabla\mathbf{u}\|_{L^2}\|\nabla\mathbf{u}\|_{L^6}\mathrm{d}t\notag\\
  &\leq C(\hat{\rho})C_0^\frac{5\alpha+7}{3(7\alpha+5)}
  \int_0^{\sigma(T)}\sigma \|\nabla\mathbf{u}\|_{L^2}\big(\|\sqrt{\rho}\|_{L^6}\|\sqrt{\rho}\dot{\mathbf{u}}\|_{L^2}+\|P\|_{L^6}\big)\mathrm{d}t\notag\\
  &\leq C(\hat{\rho})C_0^\frac{5\alpha+7}{3(7\alpha+5)}
   \int_0^{\sigma(T)}\sigma\big( \|\nabla\mathbf{u}\|_{L^2}^2+\|\rho\|_{L^3}\|\sqrt{\rho}\dot{\mathbf{u}}\|_{L^2}^2+\|P\|_{L^6}\|\nabla\mathbf{u}\|_{L^2}\big)\mathrm{d}t\notag\\
   &\leq C(\hat{\rho})C_0^\frac{2(\alpha+2)}{7\alpha+5}+C_1(\hat{\rho})C_0^\frac{2(5\alpha+7)}{3(7\alpha+5)}A_1(T),
   \\
  \int_{\sigma(T)}^T\int\sigma P|\nabla\mathbf{u}|^2\mathrm{d}\mathbf{x}\mathrm{d}t
  &\leq  \int_{\sigma(T)}^T\|P\|_{L^4}\|\nabla\mathbf{u}\|_{L^2}\|\nabla\mathbf{u}\|_{L^4}\mathrm{d}t\notag\\
  &\leq C\int_{\sigma(T)}^T\|P\|_{L^4}\|\nabla\mathbf{u}\|_{L^2}\big(\|\sqrt{\rho}\|_{L^4}\|\sqrt{\rho}\dot{\mathbf{u}}\|_{L^2}+\|P\|_{L^4}\big)\mathrm{d}t\notag\\
  &\leq
   C\int_{\sigma(T)}^T\big( \|\nabla\mathbf{u}\|_{L^2}^2+\|P\|_{L^4}^2\|\rho\|_{L^2}\|\sqrt{\rho}\dot{\mathbf{u}}\|_{L^2}^2+\|P\|_{L^4}^4\big)\mathrm{d}t\notag\\
   &\leq CC_0+C_2(\hat{\rho})C_0^\frac{5\alpha+7}{7\alpha+5}A_1(T)+\int_{\sigma(T)}^T\|P\|_{L^4}^4\mathrm{d}t.\label{3.19}
\end{align}
Putting \eqref{3.18} and \eqref{3.19} into \eqref{3.17} implies that
\begin{align*}
  \sup_{t\in[0,T]}\big(\sigma B(t)\big)+\int_0^T\int\sigma\rho |\dot{\mathbf{u}}|^2\mathrm{d}\mathbf{x}\mathrm{d}t&\leq
  C(\hat{\rho})C_0^\frac{2(\alpha+2)}{7\alpha+5}+2C_1(\hat{\rho})C_0^\frac{2(5\alpha+7)}{3(7\alpha+5)}A_1(T)
\notag\\&\quad
  +C\int_0^T\int \sigma|\nabla\mathbf{u}|^3\mathrm{d}\mathbf{x}\mathrm{d}t
+\int_{\sigma(T)}^T\|P\|_{L^4}^4\mathrm{d}t,
\end{align*}
which along with \eqref{3.1} and \eqref{3.16} immediately yields \eqref{3.8} provided
\begin{equation*}
  C_0\leq\varepsilon_1\triangleq \min\left\{1,\bigg(\frac{\mu}{8C_1(\hat{\rho})}\bigg)^\frac{3(7\alpha+5)}{2(5\alpha+7)}\right\}.
\end{equation*}

\textbf{Estimate for $A_2(T)$}.
Operating $\sigma^3\dot{u}^j[\partial/\partial t+\divv({\mathbf{u}}\cdot)]$ on $\eqref{3.11}^j$, summing all the equalities with respect to $j$, and integrating the resultant over $\mathbb{R}^2_+$,
we obtain that
\begin{align}\label{3.20}
&\frac{1}{2}\frac{\mathrm{d}}{\mathrm{d}t}\int\sigma^3\rho|\dot{\mathbf{u}}|^2\mathrm{d}\mathbf{x}
-\frac{3}{2}\sigma^2\sigma'\int\rho|\dot{\mathbf{u}}|^2\mathrm{d}\mathbf{x}\notag\\
&=-\int\sigma^3\dot{u}^j[\partial_j P_{t}+\divv(\mathbf{u}\partial_{j}P)]\mathrm{d}\mathbf{x}
+\mu\int\sigma^3\dot{u}^j\big[\Delta u_{t}^j+\divv\big( \mathbf{u}\Delta u^{j}\big)\big]\mathrm{d}\mathbf{x}\notag\\
&\quad+(\mu+\lambda)\int\sigma^3\dot{u}^j[\partial_j\divv\mathbf{u}_t +\divv(\mathbf{u}\partial_j\divv\mathbf{u})]\mathrm{d}\mathbf{x}\triangleq \sum_{i=1}^{3}\mathcal{I}_i.
\end{align}
It follows from \eqref{3.7} and Cauchy--Schwarz inequality that
\begin{align}\label{3.21}
\mathcal{I}_1&=\sigma^3\int P_{t}\divv\dot{\mathbf{u}}\mathrm{d}\mathbf{x}-\sigma^3\int\dot{\mathbf{u}}\cdot\nabla\divv(P\mathbf{u})\mathrm{d}\mathbf{x}
+\sigma^3\int\dot{u}^j\divv(P\partial_{j}\mathbf{u})\mathrm{d}\mathbf{x}\notag\\
&=\sigma^3\int \big(P_{t}+\divv(P\mathbf{u})\big)\divv\dot{\mathbf{u}}\mathrm{d}\mathbf{x}
+\sigma^3\int(\dot{\mathbf{u}}\cdot\nabla)\mathbf{u}\cdot\nabla P\mathrm{d}\mathbf{x}
+\sigma^3\int P\dot{\mathbf{u}}\cdot\nabla\divv\mathbf{u}\mathrm{d}\mathbf{x}\notag\\
&=-\sigma^3\int(\gamma-1)P\divv\mathbf{u}\divv\dot{\mathbf{u}}\mathrm{d}\mathbf{x}
-\sigma^3\int P\partial_i\dot{u}^j\partial_ju^i \mathrm{d}\mathbf{x}
\notag\\
&\leq \frac{\mu}{8}\sigma^3\|\nabla\dot{\mathbf{u}}\|_{L^2}^2
+C\sigma^3\|\nabla\mathbf{u}\|_{L^4}^4+C\sigma^3\|P\|_{L^4}^4.
\end{align}
Using integration by parts, we have that
\begin{align}\label{3.22}
\mathcal{I}_2&=\mu\sigma^3\int\dot{u}^j\big[\Delta \dot{u}^j-\Delta(\mathbf{u}\cdot\nabla u^j)+\divv\big( \mathbf{u}\Delta u^{j}\big)\big]\mathrm{d}\mathbf{x}\notag\\
&=\mu\sigma^3\int\big[-|\nabla\dot{\mathbf{u}}|^2+\dot{u}^j_i(u^ku_k^j)_i
-\dot{u}^j_i(u^ku^j_i)_k-\dot{u}^j(u^k_iu^j_i)_k\big]\mathrm{d}\mathbf{x}\notag\\
&=\mu\sigma^3\int\big[-|\nabla\dot{\mathbf{u}}|^2+\dot{u}^j_i(u^ku_k^j)_i
-\dot{u}^j_i(u^ku^j_i)_k+\dot{u}_k^j(u^k_iu^j_i)\big]\mathrm{d}\mathbf{x}\notag\\
&\leq-\frac{3\mu}{4}\sigma^3\|\nabla\dot{\mathbf{u}}\|_{L^2}^2+C\sigma^3\|\nabla\mathbf{u}\|_{L^4}^4,\\
\mathcal{I}_3
&=(\mu+\lambda)\sigma^3\int\dot{u}^j[\partial_j\divv\mathbf{u}_t+ \partial_j\divv(\mathbf{u}\divv\mathbf{u})-\divv(\partial_j\mathbf{u}\divv\mathbf{u})]\mathrm{d}\mathbf{x}\notag\\
&=-(\mu+\lambda)\sigma^3\int\divv\dot{\mathbf{u}}[\divv\mathbf{u}_t+\divv(\mathbf{u}\divv\mathbf{u})]\mathrm{d}\mathbf{x}
-(\mu+\lambda)\sigma^3\int\dot{u}^j\divv(\partial_j\mathbf{u}\divv\mathbf{u})\mathrm{d}\mathbf{x}\notag\\
&=-(\mu+\lambda)\sigma^3\int\big[\divv\dot{\mathbf{u}}
\big(\divv\dot{\mathbf{u}}-\partial_iu^j\partial_ju^i+(\divv\mathbf{u})^2\big)
-\partial_j\mathbf{u}\cdot\nabla\dot{u}^j\divv\mathbf{u}\big]\mathrm{d}\mathbf{x}\notag\\
&\leq-\frac{\mu+\lambda}{2}\sigma^3\|\divv\dot{\mathbf{u}}\|_{L^2}^2
+\frac{\mu}{8}\sigma^3\|\nabla\dot{\mathbf{u}}\|_{L^2}^2+C\sigma^3\|\nabla\mathbf{u}\|_{L^4}^4.\label{3.23}
\end{align}
Thus, substituting \eqref{3.21}--\eqref{3.23} into \eqref{3.20}, one finds that
\begin{align*}
\frac{\mathrm{d}}{\mathrm{d}t}\int\sigma^3\rho|\dot{\mathbf{u}}|^2\mathrm{d}\mathbf{x}
+\mu\sigma^3\|\nabla\dot{\mathbf{u}}\|_{L^2}^2+(\mu+\lambda)\sigma^3\|\divv\dot{\mathbf{u}}\|_{L^2}^2\leq C\sigma^2\sigma'\int\rho|\dot{\mathbf{u}}|^2\mathrm{d}\mathbf{x}+C\sigma^3\|P\|_{L^4}^4
+C\sigma^3\|\nabla\mathbf{u}\|_{L^4}^4.
\end{align*}
Integrating the above inequality over $(0,T)$ implies that
\begin{align}\label{3.24}
&\sup_{t\in[0,T]}\int\sigma^3\rho|\dot{\mathbf{u}}|^2\mathrm{d}\mathbf{x}
 +\mu\int_0^T\sigma^3\|\nabla\dot{\mathbf{u}}\|_{L^2}^2\mathrm{d}t
 +(\mu+\lambda)\int_0^T\sigma^3\|\divv\dot{\mathbf{u}}\|_{L^2}^2\mathrm{d}t\notag\\
 &\leq CA_1(T)
 +C\int_0^T\sigma^3\|P\|_{L^4}^4\mathrm{d}t
 +C\int_0^T\sigma^3\|\nabla\mathbf{u}\|_{L^4}^4\mathrm{d}t.
\end{align}

It remains to control the last term in \eqref{3.24}. Following the decomposition of the velocity field as in Lemma \ref{l2.8}, we have that
\begin{align}\label{3.25}
 \|\nabla\mathbf{u}\|_{L^4}^4&\leq C\|\nabla\mathbf{w}_1\|_{L^4}^4+C\|\nabla\mathbf{w}_2\|_{L^4}^4\notag\\
 &\leq C\|\nabla\mathbf{w}_1\|_{L^2}\|\nabla\mathbf{w}_1\|_{L^6}^3+C\|P\|_{L^4}^4\notag\\
 &\leq C(\|\nabla\mathbf{u}\|_{L^2}+\|\nabla\mathbf{w}_2\|_{L^2})\|\nabla^2\mathbf{w}_1\|_{L^{\frac32}}^3+C\|P\|_{L^4}^4\notag\\
 &\leq C(\|\nabla\mathbf{u}\|_{L^2}+\|P\|_{L^2})\|\rho\dot{\mathbf{u}}\|_{L^{\frac32}}^3
 +C\|P\|_{L^4}^4\notag\\
 &\leq C(\|\nabla\mathbf{u}\|_{L^2}+\|P\|_{L^2})\|\sqrt{\rho}\|_{L^6}^3\|\sqrt{\rho}\dot{\mathbf{u}}\|_{L^2}^3
 +C\|P\|_{L^4}^4\notag\\
 &\leq C(\hat{\rho})(\|\nabla\mathbf{u}\|_{L^2}+1)\|\sqrt{\rho}\dot{\mathbf{u}}\|_{L^2}^3
 +C\|P\|_{L^4}^4
\end{align}
provided $C_0\leq1$. Hence, one deduces from \eqref{3.1} and \eqref{3.2} that
\begin{align}\label{3.26}
 \int_0^T\sigma^3\|\nabla\mathbf{u}\|_{L^4}^4\mathrm{d}t
 &\leq
 C(\hat{\rho})\sup_{t\in[0,T]}\big(\sigma^\frac12\|\nabla\mathbf{u}\|_{L^2}\big)
 \sup_{t\in[0,T]}\big(\sigma^\frac32\|\sqrt{\rho}\dot{\mathbf{u}}\|_{L^2}\big)\int_0^T\sigma\|\sqrt{\rho}\dot{\mathbf{u}}\|_{L^2}^2\mathrm{d}t
 \notag\\&\quad
 +C(\hat{\rho})
 \sup_{t\in[0,T]}\big(\sigma^\frac32\|\sqrt{\rho}\dot{\mathbf{u}}\|_{L^2}\big)\int_0^T\sigma^\frac32\|\sqrt{\rho}\dot{\mathbf{u}}\|_{L^2}^2\mathrm{d}t
 +C\int_0^T\sigma^3\|P\|_{L^4}^4\mathrm{d}t\notag\\
 &\leq C(\hat{\rho})A_2^\frac12(T)\Big(A_1^\frac32(T)+A_1(T)\Big)
 +C\int_0^T\sigma^3\|P\|_{L^4}^4\mathrm{d}t,
\end{align}
which combined with \eqref{3.24} gives \eqref{3.9}.
\end{proof}

On the basis of the preliminary bounds and the \textit{a priori} hypothesis \eqref{3.4}, we proceed to derive the key estimates for $A_1(T)$ and $A_2(T)$.

\begin{lemma}\label{l3.3}
Under the assumption \eqref{3.4}, there exists a positive constant $\varepsilon_2$ depending only on $\alpha, \hat{\rho}, a, \gamma$, $\mu, \lambda$, and $K$ such that
\begin{align}\label{3.27}
A_1(T)+A_2(T)\leq  C_0^\frac{\alpha+2}{7\alpha+5}
\end{align}
provided $C_0\leq\varepsilon_2$ and $\frac{3\mu}{3\mu+\lambda}\leq \frac{1}{2\Lambda(4)}$.
\end{lemma}
\begin{proof}
If $C_0\leq 1$, it follows from \eqref{3.4} and \eqref{3.6} that
\begin{align}\label{3.28}
  \int_0^T\|\nabla \mathbf{u}\|_{L^2}^4\mathrm{d}t
  &\leq\sup_{t\in[0,\sigma(T)]}\|\nabla \mathbf{u}\|_{L^2}^2\int_0^{\sigma(T)}\|\nabla \mathbf{u}\|_{L^2}^2\mathrm{d}t
  +\sup_{t\in[\sigma(T),T]}\big(\sigma\|\nabla \mathbf{u}\|_{L^2}^2\big)\int_{\sigma(T)}^T\|\nabla \mathbf{u}\|_{L^2}^2\mathrm{d}t\notag\\
  &\leq CC_0A_3(\sigma(T))+CC_0A_1(T)\notag\\
  &\leq C(K)C_0.
\end{align}

From \eqref{3.4} and Lemma \ref{l3.2}, we have that
\begin{align}\label{3.29}
A_1(T)+A_2(T)\leq CC_0^\frac{2(\alpha+2)}{7\alpha+5}+C\int_0^T \sigma\|\nabla\mathbf{u}\|_{L^3}^3\mathrm{d}t
 +C\int_{0}^T\sigma^3\|P\|_{L^4}^4\mathrm{d}t
\end{align}
provided $C_0\leq\varepsilon_1$. Using \eqref{3.28} and Lemma \ref{l2.8}, we obtain that
\begin{align}\label{3.30}
 \int_0^{\sigma(T)}\sigma\|\nabla\mathbf{u}\|_{L^3}^3\mathrm{d}t
 &\leq
 \int_0^{\sigma(T)}\sigma\|\nabla\mathbf{u}\|_{L^2}^\frac32\|\nabla\mathbf{u}\|_{L^6}^\frac32\mathrm{d}t\notag\\
 &\leq C(\hat{\rho})
 \int_0^{\sigma(T)}\sigma\|\nabla\mathbf{u}\|_{L^2}^\frac32
 \big(\|\sqrt{\rho}\|_{L^6}\|\sqrt{\rho}\dot{\mathbf{u}}\|_{L^2}+\|P\|_{L^6}\big)^\frac32\mathrm{d}t\notag\\
 &\leq C(\hat{\rho})\sup_{t\in[0,\sigma(T)]}\Big(\sigma^\frac14\|\nabla \mathbf{u}\|_{L^2}^\frac12\Big)
 \int_0^{\sigma(T)}\big(\sigma\|\sqrt{\rho}\dot{\mathbf{u}}\|_{L^2}^2+\|\nabla \mathbf{u}\|_{L^2}^4\big)\mathrm{d}t\notag\\
 &\quad+ C(\hat{\rho})\sup_{t\in[0,\sigma(T)]} \|P\|_{L^6}^\frac32\int_0^{\sigma(T)}\|\nabla\mathbf{u}\|_{L^2}^{\frac32}\mathrm{d}t\notag\\
 &\leq C_1(\hat{\rho},K)A_1^\frac14(T)(A_1(T)+C_0)+C(\hat{\rho})C_0^\frac34.
\end{align}
Moreover, one gets from \eqref{3.26} that
\begin{align*}
  \int_{\sigma(T)}^T\sigma\|\nabla\mathbf{u}\|_{L^3}^3\mathrm{d}t
  &\leq
 \int_{\sigma(T)}^T\|\nabla\mathbf{u}\|_{L^2}\|\nabla\mathbf{u}\|_{L^4}^2\mathrm{d}t\notag\\
  &\leq C\int_{\sigma(T)}^T\big(\|\nabla \mathbf{u}\|_{L^2}^2+\sigma^3\|\nabla \mathbf{u}\|_{L^4}^4\big)\mathrm{d}t\notag\\
 &\leq CC_0+C(\hat{\rho})A_2^\frac12(T)
 \Big(A_1^\frac32(T)+A_1(T)\Big)+C\int_0^T\sigma^3\|P\|_{L^4}^4\mathrm{d}t,
\end{align*}
which combined with \eqref{3.4}, \eqref{3.29}, and \eqref{3.30} implies that
\begin{align}\label{3.31}
A_1(T)+A_2(T)\leq 2C_1(\hat{\rho},K)C_0^\frac{5(\alpha+2)}{4(7\alpha+5)}
 +C\int_{0}^T\sigma^3\|P\|_{L^4}^4\mathrm{d}t.
\end{align}

We now turn to estimate the pressure term in \eqref{3.31}. Recall the decomposition of $F$ in \eqref{2.9}:
\begin{align}\label{3.32}
(2\mu+\lambda)\divv\mathbf{u}-P=
F= \widetilde{F}+\frac{2\mu}{3\mu+\lambda}F_P.
\end{align}
To proceed, we employ the equation \eqref{3.7}:
\begin{equation*}
P_t+\divv(P\mathbf{u})+(\gamma-1)P\divv \mathbf{u}=0.
\end{equation*}
After multiplying it by $3\sigma^3P^2$ and integrating by parts, we obtain from Young's inequality that
\begin{align}\label{ef1}
\frac{3\gamma-1}{2\mu+\lambda}\sigma^3\|P\|_{L^4}^4&=
-\frac{\mathrm{d}}{\mathrm{d}t}\int \sigma^3P^3\mathrm{d}\mathbf{x}
+3\sigma^2\sigma'\|P\|_{L^3}^3
-\frac{3\gamma-1}{2\mu+\lambda}\int \sigma^3FP^3\mathrm{d}\mathbf{x}\notag\\
&\leq -\frac{\mathrm{d}}{\mathrm{d}t}\int \sigma^3P^3\mathrm{d}\mathbf{x}
+3\sigma^2\sigma'\|P\|_{L^3}^3
+\frac{3(3\gamma-1)}{4(2\mu+\lambda)}\sigma^3\|P\|_{L^4}^4
+\frac{3\gamma-1}{4(2\mu+\lambda)}\sigma^3\|F\|_{L^4}^4.
\end{align}
Integrating the above inequality over $(0,T)$ gives that
\begin{align*}
  \int_{0}^{T}\sigma^3\|P\|_{L^{4}}^{4}\mathrm{d}t&\leq C\sup_{ t\in[0, T]}\|P\|_{L^3}^3
  +C\int_{0}^{\sigma(T)}\|P\|_{L^3}^3\mathrm{d}t+\int_{0}^{T}\sigma^3\|F\|_{L^{4}}^{4}\mathrm{d}t\notag\\
  &\leq CC_0^\frac{5\alpha+7}{7\alpha+5}+
  8\int_{0}^{T}\sigma^3\|\widetilde{F}\|_{L^{4}}^{4}\mathrm{d}t
  +8\left(\frac{2\mu}{3\mu+\lambda}\right)^4\int_{0}^{T}\sigma^3\|F_P\|_{L^4}^4\mathrm{d}t\notag\\
  &\leq CC_0^\frac{5\alpha+7}{7\alpha+5}+
  8\int_{0}^{T}\sigma^3\|\widetilde{F}\|_{L^{4}}^{4}\mathrm{d}t
  +8\left(\frac{2\mu\Lambda(4)}{3\mu+\lambda}\right)^4\int_{0}^{T}\sigma^3\|P\|_{L^4}^4\mathrm{d}t
\end{align*}
due to \eqref{3.32}, \eqref{3.10}, and \eqref{2.16}. Then it follows from \eqref{ker} that
\begin{align}\label{3.33}
\int_{0}^{T}\sigma^3\|P\|_{L^{4}}^{4}\mathrm{d}t&\leq CC_0^\frac{5\alpha+7}{7\alpha+5}
+C\int_{0}^{T}\sigma^3\|\widetilde{F}\|_{L^{4}}^{4}\mathrm{d}t\notag\\
&\leq CC_0^\frac{5\alpha+7}{7\alpha+5}
+C\int_{0}^{T}\sigma^3\|\widetilde{F}\|_{L^{2}}
\|\widetilde{F}\|_{L^{6}}^{3}\mathrm{d}t\notag\\
&\leq CC_0^\frac{5\alpha+7}{7\alpha+5}
+C\int_{0}^{T}\sigma^3\|\widetilde{F}\|_{L^{2}}\|\nabla\widetilde{F}\|_{L^{\frac32}}^{3}\mathrm{d}t\notag\\
&\leq CC_0^\frac{5\alpha+7}{7\alpha+5}
+C(\hat{\rho})\int_{0}^{T}\sigma^3\big(\|\nabla\mathbf{u}\|_{L^{2}}
+\|P\|_{L^{2}}\big)
\|\sqrt{\rho}\|_{L^6}^3\|\sqrt{\rho}\dot{\mathbf{u}}\|_{L^2}^3
\mathrm{d}t\notag\\
&\leq CC_0^\frac{5\alpha+7}{7\alpha+5}
+C(\hat{\rho})\sup_{t\in[0,T]}\big(\sigma^\frac12\|\nabla\mathbf{u}\|_{L^{2}}+\|P\|_{L^{2}}\big)
\sup_{t\in[0,T]}\big(\sigma^\frac32\|\sqrt{\rho}\dot{\mathbf{u}}\|_{L^2}\big)
\int_{0}^{T}\sigma\|\sqrt{\rho}\dot{\mathbf{u}}\|_{L^2}^2\mathrm{d}t\notag\\
&\leq  CC_0^\frac{5\alpha+7}{7\alpha+5}+
C_3(\hat{\rho})\Big(A_1^\frac12(T)+1\Big)
A_1(T)A_2^\frac12(T)
\end{align}
provided $\frac{3\mu}{3\mu+\lambda}\leq \frac{1}{2\Lambda(4)}$,
where in the fourth inequality we have used
\begin{gather}\label{3.34}
\|\widetilde{F}\|_{L^2}
=\Big\|(2\mu+\lambda)\divv\mathbf{u}-P-\frac{2\mu}{3\mu+\lambda}F_P\Big\|_{L^2}
\leq C\big(\|\nabla\mathbf{u}\|_{L^{2}}+\|P\|_{L^{2}}\big),\\
\|\widetilde{F}\|_{L^6}\leq C
\|\nabla\widetilde{F}\|_{L^\frac32}\leq
 C\|\sqrt{\rho}\|_{L^6}\|\sqrt{\rho}\dot{\mathbf{u}}\|_{L^2}.\label{3.35}
\end{gather}

As a consequence, one deduces from \eqref{3.4}, \eqref{3.31}, and \eqref{3.33} that
\begin{align*}
  A_1(T)+A_2(T)&\leq 2C_1(\hat{\rho},K)C_0^\frac{5(\alpha+2)}{4(7\alpha+5)}
 +C_3(\hat{\rho})\Big(A_1^\frac12(T)+1\Big)A_1(T)A_2^\frac12(T)\notag\\
 &\leq 2C_1(\hat{\rho},K)C_0^\frac{\alpha+2}{4(7\alpha+5)}
 C_0^\frac{\alpha+2}{7\alpha+5}
 +2C_3(\hat{\rho})C_0^\frac{\alpha+2}{2(7\alpha+5)}
 C_0^\frac{\alpha+2}{7\alpha+5},
\end{align*}
which yields \eqref{3.27} as long as
\begin{equation*}
  C_0\leq\varepsilon_2\triangleq \min\left\{\varepsilon_1,
  \bigg(\frac{1}{4C_1(\hat{\rho},K)}\bigg)^\frac{4(7\alpha+5)}{\alpha+2},
  \bigg(\frac{1}{4C_3(\hat{\rho})}\bigg)^\frac{2(7\alpha+5)}{\alpha+2}\right\}.\tag*{\qedhere}
\end{equation*}
\end{proof}

Next, we bound the auxiliary functional $A_3(T)$ and $\sigma\|\sqrt{\rho}\dot{\mathbf{u}}\|_{L^2}^2$
for small time.

\begin{lemma}\label{l3.4}
Under the assumption \eqref{3.4}, there exist positive constants $K\geq M^2$, $C(K)$, and $\varepsilon_3$ depending only on $\alpha, \hat{\rho}, M, a, \gamma, \mu$, and $\lambda$ such that
\begin{gather}\label{3.36}
  A_3(\sigma(T))=\sup_{t\in[0,\sigma(T)]}\|\nabla \mathbf{u}\|_{L^2}^2
   +\int_0^{\sigma(T)}\|\sqrt{\rho}\dot{\mathbf{u}}\|_{L^2}^2\mathrm{d}t
  \leq 2K, \\
  \sup_{t\in[0,\sigma(T)]}\big(\sigma\|\sqrt{\rho} \dot{\mathbf{u}}\|_{L^2}^2\big)
   +\int_0^{\sigma(T)}\sigma\|\nabla \dot{\mathbf{u}}\|_{L^2}^2\mathrm{d}t\leq C(K)\label{3.37}
\end{gather}
provided $C_0\leq\varepsilon_3$.
\end{lemma}
\begin{proof}
Multiplying \eqref{3.11} by $\dot{\mathbf{u}}$ and integrating the resultant over $\mathbb{R}^2_+$, one has that
\begin{align}\label{3.38}
\int\rho |\dot{\mathbf{u}}|^2\mathrm{d}\mathbf{x}=
\int[-\dot{\mathbf{u}}\cdot\nabla P+\mu\Delta\mathbf{u}\cdot\dot{\mathbf{u}}
+(\mu+\lambda)\dot{\mathbf{u}}\cdot\nabla\divv \mathbf{u}]\mathrm{d}\mathbf{x}.
\end{align}
Following the same argument as in \eqref{3.12}--\eqref{3.14}, with $\sigma$ replaced by $1$, we obtain that
\begin{align}\label{3.39}
\frac{\mathrm{d}}{\mathrm{d}t}B(t)+\int\rho |\dot{\mathbf{u}}|^2\mathrm{d}\mathbf{x}
&\leq C\|\nabla\mathbf{u}\|^2_{L^2}+C\|\nabla\mathbf{u}\|^3_{L^3}
+C\int P|\nabla\mathbf{u}|^2\mathrm{d}\mathbf{x},
\end{align}
where $B(t)$ is given by \eqref{3.16}. If $C_0\leq1$, integrating \eqref{3.39} over $(0,\sigma(T))$ shows that
\begin{align}\label{3.40}
  &\sup_{t\in[0,\sigma(T)]}B(t)+\int_0^{\sigma(T)}\int\rho |\dot{\mathbf{u}}|^2\mathrm{d}\mathbf{x}\mathrm{d}t
 \leq CM^2
+C\int_0^{\sigma(T)}\|\nabla\mathbf{u}\|^3_{L^3}\mathrm{d}t
+C\int_0^{\sigma(T)}\int P|\nabla\mathbf{u}|^2\mathrm{d}\mathbf{x}\mathrm{d}t
\end{align}
due to \eqref{1.9} and \eqref{3.6}.

It remains to estimate the last two terms in \eqref{3.40}.
From \eqref{3.3}, \eqref{3.10}, and Lemma \ref{l2.8}, one sees that
\begin{align}\label{3.41}
 \int_0^{\sigma(T)}\|\nabla\mathbf{u}\|_{L^3}^3\mathrm{d}t
 &\leq
 \int_0^{\sigma(T)}\|\nabla\mathbf{u}\|_{L^2}^\frac32\|\nabla\mathbf{u}\|_{L^6}^\frac32\mathrm{d}t\notag\\
 &\leq C
 \int_0^{\sigma(T)}\|\nabla\mathbf{u}\|_{L^2}^\frac32
 \big(\|\sqrt{\rho}\|_{L^6}\|\sqrt{\rho}\dot{\mathbf{u}}\|_{L^2}+\|P\|_{L^6}\big)^\frac32\mathrm{d}t\notag\\
 &\leq C(\hat{\rho})\sup_{t\in[0,\sigma(T)]}\|\nabla \mathbf{u}\|_{L^2}^{\frac32}
 \sup_{t\in[0,\sigma(T)]} \|\rho\|_{L^3}^\frac34
 \int_0^{\sigma(T)}\|\sqrt{\rho}\dot{\mathbf{u}}\|_{L^2}^{\frac32}\mathrm{d}t\notag\\
 &\quad+ C(\hat{\rho})\sup_{t\in[0,\sigma(T)]} \|P\|_{L^6}^\frac32\int_0^{\sigma(T)}\|\nabla\mathbf{u}\|_{L^2}^{\frac32}\mathrm{d}t\notag\\
 &\leq C_4(\hat{\rho})C_0^\frac{5\alpha+7}{4(7\alpha+5)}
 A_3^\frac32(\sigma(T))+C(\hat{\rho})C_0^\frac34,
 \\
 \int_0^{\sigma(T)}\int P|\nabla\mathbf{u}|^2\mathrm{d}\mathbf{x}\mathrm{d}t
  &\leq  \int_0^{\sigma(T)}\|P\|_{L^3}\|\nabla\mathbf{u}\|_{L^2}\|\nabla\mathbf{u}\|_{L^6}\mathrm{d}t\notag\\
  &\leq C(\hat{\rho})C_0^\frac{5\alpha+7}{3(7\alpha+5)}
  \int_0^{\sigma(T)} \|\nabla\mathbf{u}\|_{L^2}\big(\|\sqrt{\rho}\|_{L^6}\|\sqrt{\rho}\dot{\mathbf{u}}\|_{L^2}+\|P\|_{L^6}\big)\mathrm{d}t\notag\\
 &\leq C(\hat{\rho})C_0^\frac{5\alpha+7}{2(7\alpha+5)}
 A_3(\sigma(T))+C(\hat{\rho})C_0^\frac12.\label{3.42}
\end{align}
Thus, it follows from \eqref{3.40}--\eqref{3.42} and \eqref{3.3} that
\begin{align*}
   A_3(\sigma(T))
 &\leq C(\hat{\rho},M)+C_4(\hat{\rho})C_0^\frac{5\alpha+7}{4(7\alpha+5)}
 A_3^\frac32(\sigma(T))+C(\hat{\rho})C_0^\frac{5\alpha+7}{2(7\alpha+5)}
 A_3(\sigma(T))\notag\\
&\leq K+2C_4(\hat{\rho})C_0^\frac{5\alpha+7}{4(7\alpha+5)}
 A_3^\frac32(\sigma(T))
\end{align*}
for $K\triangleq C(\hat{\rho},M)$, which immediately yields \eqref{3.36} provided
\begin{equation*}
  C_0\leq\varepsilon_{3}\triangleq \min\left\{\varepsilon_2,
  \left(\frac{1}{18KC_4(\hat{\rho})}\right)^\frac{4(7\alpha+5)}{5\alpha+7}\right\}.
\end{equation*}

We now turn to the proof of \eqref{3.37}. Operating $\sigma\dot{u}^j[\partial/\partial t+\divv({\mathbf{u}}\cdot)]$ on $\eqref{3.11}^j$, summing all the equalities with respect to $j$, and integrating the resultant over $\mathbb{R}^2_+$, we deduce from \eqref{3.20}--\eqref{3.23} (with $\sigma^3$ replaced by $\sigma$) that
\begin{align*}
\frac{\mathrm{d}}{\mathrm{d}t}\int\sigma\rho|\dot{\mathbf{u}}|^2\mathrm{d}\mathbf{x}
+\mu\sigma\|\nabla\dot{\mathbf{u}}\|_{L^2}^2+(\mu+\lambda)\sigma\|\divv\dot{\mathbf{u}}\|_{L^2}^2\leq C\sigma'\int\rho|\dot{\mathbf{u}}|^2\mathrm{d}\mathbf{x}
+C\sigma\|\nabla\mathbf{u}\|_{L^4}^4+C\sigma\|P\|_{L^4}^4.
\end{align*}
Integrating the above inequality over $(0,\sigma(T))$ implies that
\begin{align}\label{3.43}
\sup_{t\in[0,\sigma(T)]}\int\sigma\rho|\dot{\mathbf{u}}|^2\mathrm{d}\mathbf{x}
 +\int_0^{\sigma(T)}\sigma\|\nabla\dot{\mathbf{u}}\|_{L^2}^2\mathrm{d}t&\leq CA_3(\sigma(T))+C\int_0^{\sigma(T)}\sigma\|\nabla\mathbf{u}\|_{L^4}^4\mathrm{d}t
 +C\int_0^{\sigma(T)}\sigma\|P\|_{L^4}^4\mathrm{d}t\notag\\
 &\leq C(K)+C\int_0^{\sigma(T)}\sigma\|\nabla\mathbf{u}\|_{L^4}^4\mathrm{d}t
 +C\int_0^{\sigma(T)}\sigma\|P\|_{L^4}^4\mathrm{d}t.
\end{align}
Arguing as in \eqref{3.25}, \eqref{3.26}, and \eqref{3.33}, we obtain that
\begin{align*}
&\int_0^{\sigma(T)}\sigma\|\nabla\mathbf{u}\|_{L^4}^4\mathrm{d}t
 +\int_0^{\sigma(T)}\sigma\|P\|_{L^4}^4\mathrm{d}t\notag\\
 &\leq CC_0^\frac{5\alpha+7}{7\alpha+5}
 +C(\hat{\rho})\int_0^{\sigma(T)}\sigma
 (\|\nabla\mathbf{u}\|_{L^2}+1)\|\sqrt{\rho}\dot{\mathbf{u}}\|_{L^2}^3\mathrm{d}t
+C(\hat{\rho})\int_{0}^{\sigma(T)}\sigma\big(\|\nabla\mathbf{u}\|_{L^{2}}+\|P\|_{L^{2}}\big)
\|\sqrt{\rho}\dot{\mathbf{u}}\|_{L^2}^3
\mathrm{d}t\notag\\
&\leq CC_0^\frac{5\alpha+7}{7\alpha+5}
+C(\hat{\rho})\sup_{t\in[0,\sigma(T)]}\big(\sigma^\frac12\|\nabla\mathbf{u}\|_{L^{2}}+\|P\|_{L^{2}}+1\big)
\sup_{t\in[0,\sigma(T)]}\big(\sigma^\frac12\|\sqrt{\rho}\dot{\mathbf{u}}\|_{L^2}\big)
\int_{0}^{\sigma(T)}\|\sqrt{\rho}\dot{\mathbf{u}}\|_{L^2}^2\mathrm{d}t\notag\\
&\leq CC_0^\frac{5\alpha+7}{7\alpha+5}+C(\hat{\rho})A_3(\sigma(T))
\sup_{t\in[0,\sigma(T)]}\big(\sigma\|\sqrt{\rho}\dot{\mathbf{u}}\|_{L^2}^2\big)^\frac12,
\end{align*}
which along with \eqref{3.43} and \eqref{3.36} yields \eqref{3.37}.
\end{proof}

The following spatial weighted estimate plays an important role in bounding $\|\rho\|_{L^\theta}$.

\begin{lemma}\label{l3.5}
Let $\bar{x}$ and $\alpha\in(1,2)$ be as in \eqref{1.10}.
There exists a positive constant $C$ depending only on $\alpha, \hat{\rho}, M, a,
\gamma, \mu, \lambda$, and $\eta_0$ such that, for any $0<t<T$,
\begin{equation}\label{3.44}
  \sup_{s\in[0,t]}\|\bar{x}^\alpha\rho\|_{L^1}\leq C(1+t)^4.
\end{equation}
\end{lemma}
\begin{proof}
Multiplying $\eqref{a1}_1$ by $(1+|\mathbf{x}|^2)^\frac12$ and integrating the resulting equality over $\mathbb{R}^2_+$, we obtain that
\begin{equation*}
  \frac{\mathrm{d}}{\mathrm{d}t}\int\rho(1+|\mathbf{x}|^2)^\frac12\mathrm{d}\mathbf{x}
\leq C\int\rho|\mathbf{u}|\mathrm{d}\mathbf{x}
\leq C\|\sqrt{\rho}\|_{L^2}\|\sqrt{\rho}\mathbf{u}\|_{L^2}.
\end{equation*}
Integrating the above inequality with respect to time over $(0,t)$ together with \eqref{wz} and \eqref{3.6} yields that
\begin{equation}\label{3.44*}
  \sup_{s\in[0,t]}\int\rho(1+|\mathbf{x}|^2)^\frac12\mathrm{d}\mathbf{x}\leq C(M)(1+t).
\end{equation}

For $N\geq1$, let $\psi_N\in C^\infty_0(\widetilde{B}_{2N})$ satisfy
\begin{equation}\label{3.45}
  \psi_N(\mathbf{x})=
  \begin{cases}
  1, \ \ |\mathbf{x}|\leq N, \\
  0, \ \ |\mathbf{x}|\geq 2N,
  \end{cases}
  \ \ \ 0\leq\psi_N\leq1, \ \ \ |\nabla\psi_N|\leq CN^{-1}.
\end{equation}
Set
\begin{equation*}
\mathbf{y}(t)=\delta\mathbf{x}(1+t)^{-1}\log^{-\alpha}(e+t),
\end{equation*}
with some small constant $\delta>0$ determined later. Multiplying $\eqref{a1}_1$ by $\psi_1(\mathbf{y})$ and integrating by parts, it follows from \eqref{3.44*} that
\begin{align}\label{3.46}
 \frac{\mathrm{d}}{\mathrm{d}t}\int\rho\psi_{1}(\mathbf{y})\mathrm{d}\mathbf{x}
 &=\int\rho\mathbf{y}_t\cdot\nabla_{\mathbf{y}}\psi_1\mathrm{d}\mathbf{x}
 +\delta(1+t)^{-1}\log^{-\alpha}(e+t)\int\rho\mathbf{u}\cdot\nabla_{\mathbf{y}}\psi_1\mathrm{d}\mathbf{x}\notag\\
 &\geq -\frac{C\delta}{(1+t)^{2}\log^{\alpha}(e+t)}\int\rho|\mathbf{x}|\mathrm{d}\mathbf{x}
 -\frac{C\delta}{(1+t)\log^{\alpha}(e+t)}\notag\\
&\geq -\frac{2C(M)\delta}{(1+t)\log^{\alpha}(e+t)}.
\end{align}
Thus, one deduces from \eqref{3.46} and \eqref{1.8} that
\begin{equation*}
  \int\rho\psi_{1}(\mathbf{y})\mathrm{d}\mathbf{x}
  \geq \int\rho_0\psi_{1}(\delta\mathbf{x})\mathrm{d}\mathbf{x}-C(\alpha,M)\delta
  \geq \int_{\widetilde{B}_{\eta_0}}\rho_0\mathrm{d}\mathbf{x}-C(\alpha,M)\delta
  \geq \frac{1}{4}
\end{equation*}
by choosing $\delta\triangleq\big(\eta_0+4C(\alpha,M)\big)^{-1}$. Moreover,
for $\eta_1\triangleq 2\delta^{-1}=2\eta_0+8C(\alpha,M)$, we derive that
\begin{equation}\label{3.47}
\inf_{0\leq t\leq T}\int_{\widetilde{B}_{\eta_1(1+t)\log^{\alpha}(e+t)}}\rho\mathrm{d}\mathbf{x}\geq
\inf_{0\leq t\leq T}\int\rho\psi_{1}(\mathbf{y})\mathrm{d}\mathbf{x}\geq \frac{1}{4},
\end{equation}
as the desired \eqref{1.15}.

Now we select the radius $\eta_*=\eta_1(1+t)\log^{\alpha}(e+t)$ in Lemma \ref{l2.4}, which along with Lemma \ref{l2.3}, \eqref{2.2}, \eqref{wz}, \eqref{3.4}, and \eqref{3.47} implies that
\begin{equation}\label{3.48}
  \sup_{s\in[0,t]}\|\bar{x}^{-\varsigma}\mathbf{u}\|_{L^{r/\varsigma}}\leq
  C\sup_{s\in[0,t]}\eta_*^2(1+\|\rho\|_{L^2})\big(\|\sqrt{\rho}\mathbf{u}\|_{L^2}+\|\nabla\mathbf{u}\|_{L^2}\big)
  \leq C(\hat{\rho},M)(1+t)^{3-\frac{\alpha+1}{3}}
\end{equation}
for $\varsigma\in(0,1]$, $\alpha\in(1,2)$, and $r>2$.
Then, multiplying $\eqref{a1}_1$ by $\bar{x}^\alpha$ and integrating the resultant over $\mathbb{R}^2_+$, we get from H{\"o}lder's inequality, \eqref{3.44*}, and \eqref{3.48} that
\begin{align*}
  \frac{\mathrm{d}}{\mathrm{d}t}\int\bar{x}^\alpha\rho\mathrm{d}\mathbf{x}
  &\leq C\int\rho|\mathbf{u}|\bar{x}^{\alpha-1}\log^2(e+|\mathbf{x}|^2)\mathrm{d}\mathbf{x}\notag\\
  &\leq C\big\|\big(\rho(1+|\mathbf{x}|^2)^\frac12\big)^\frac{\alpha+1}{3}\big\|_{L^\frac{3}{\alpha+1}}
  \big\|\bar{x}^{-\frac{2-\alpha}{2}}\mathbf{u}\big\|_{L^{\frac{6}{2-\alpha}}}
  \big\|\rho^{\frac{2-\alpha}{3}}\big\|_{L^{\frac{6}{2-\alpha}}}\notag\\
  &\quad\times\sup_{\mathbf{x}\in\mathbb{R}^2_+}
  \Big[(1+|\mathbf{x}|^2)^{\frac{\alpha-2}{12}}\log^{\alpha+2}(e+|\mathbf{x}|^2)\Big]
  \notag\\
  &\leq C(1+t)^3,
\end{align*}
which yields \eqref{3.44} after integrating the above inequality over $(0,t)$.
\end{proof}

Finally, by Zlotnik's inequality, one can establish the desired uniform-in-time $L^\theta$ bound for the density.

\begin{lemma}\label{l3.6}
Under the assumptions \eqref{1.11} and \eqref{3.4}, there exists a positive constant $\varepsilon_4$ depending only on $\alpha, \hat{\rho}, M, a, \gamma, \mu, \lambda$, and $\eta_0$ such that
\begin{align}\label{3.49}
\sup_{t\in[0,T]}\|\rho\|_{L^\theta}\leq\frac{7}{4}\hat{\rho}
\end{align}
provided $C_0\leq \varepsilon_4$.
\end{lemma}
\begin{proof}
From $\eqref{a1}_1$ and \eqref{1.6}, we have that
\begin{equation}\label{3.50}
\frac{\mathrm{d}}{\mathrm{d}t}\int\rho^\theta\mathrm{d}\mathbf{x}
=-(\theta-1)\int\rho^\theta\divv\mathbf{u}\mathrm{d}\mathbf{x}=
-\frac{\theta-1}{2\mu+\lambda}\int\rho^\theta(P+F)\mathrm{d}\mathbf{x},
\end{equation}
which combined with \eqref{3.32} and Lemma \ref{l2.7} leads to
\begin{align*}
\frac{\mathrm{d}}{\mathrm{d}t}\|\rho\|_{L^\theta}^\theta
+\frac{a(\theta-1)}{2\mu+\lambda}\|\rho\|_{L^{\theta+\gamma}}^{\theta+\gamma}
&\leq
\frac{\theta-1}{2\mu+\lambda}\int\rho^\theta|\widetilde{F}|\mathrm{d}\mathbf{x}
+\frac{2\mu(\theta-1)}{(3\mu+\lambda)(2\mu+\lambda)}\int\rho^\theta\big|F_P\big|\mathrm{d}\mathbf{x}\notag\\
&\leq
\frac{\theta-1}{2\mu+\lambda}\|\rho\|_{L^\theta}^\theta\|\widetilde{F}\|_{L^\infty}
+\frac{2\mu(\theta-1)}{(3\mu+\lambda)(2\mu+\lambda)}\|\rho^\theta\|_{L^{\frac{\theta+\gamma}{\theta}}}\|F_P\|_{L^{\frac{\theta+\gamma}{\gamma}}}\notag\\
&\leq\frac{\theta-1}{2\mu+\lambda}\|\rho\|_{L^\theta}^\theta\|\widetilde{F}\|_{L^\infty}
+\frac{2a\mu(\theta-1)}{(3\mu+\lambda)(2\mu+\lambda)}
\Lambda\left(\frac{9\alpha+3}{\alpha-1}\right)\|\rho\|_{L^{\theta+\gamma}}^{\theta+\gamma}.
\end{align*}
Then it follows from \eqref{1.11} that
\begin{equation}\label{3.51}
  \frac{\mathrm{d}}{\mathrm{d}t}\|\rho\|_{L^\theta}^\theta\leq
-\frac{a(\theta-1)}{3(2\mu+\lambda)}\|\rho\|_{L^{\theta+\gamma}}^{\theta+\gamma}
+\frac{\theta-1}{2\mu+\lambda}\|\rho\|_{L^\theta}^\theta\|\widetilde{F}\|_{L^\infty}.
\end{equation}
Noting that
\begin{equation*}
\|\rho\|_{L^\theta}^\theta\leq \|\rho\|_{L^\gamma}^{\frac{\gamma^2}{\theta}}
\|\rho\|_{L^{\theta+\gamma}}^{\frac{\theta^2-\gamma^2}{\theta}}\leq \big(a^{-1}(\gamma-1)C_0\big)^\frac{\gamma}{\theta}\|\rho\|_{L^{\theta+\gamma}}^{\frac{\theta^2-\gamma^2}{\theta}}
\leq a^{-\frac{\gamma}{\theta}}(\gamma-1)^\frac{\gamma}{\theta}\|\rho\|_{L^{\theta+\gamma}}^{\frac{\theta^2-\gamma^2}{\theta}}
\end{equation*}
provided $C_0\leq1$, we infer from \eqref{3.51} that
\begin{equation}\label{3.51*}
   \frac{\mathrm{d}}{\mathrm{d}t}\|\rho\|_{L^\theta}^\theta\leq
   -\frac{a\tilde{c}(\theta-1)}{3(2\mu+\lambda)}
   \big(\|\rho\|_{L^{\theta}}^{\theta}\big)^\frac{\theta}{\theta-\gamma}
   +\frac{\theta-1}{2\mu+\lambda}\|\rho\|_{L^\theta}^\theta\|\widetilde{F}\|_{L^\infty},
\end{equation}
where $\tilde{c}=\tilde{c}(\alpha,a,\gamma)=a^{-\frac{\theta^2}{\gamma(\theta-\gamma)}}
(\gamma-1)^{\frac{\theta^2}{\gamma(\theta-\gamma)}}$.

Set
\begin{align*}
 y(t)=\|\rho\|_{L^\theta}^\theta,\ \
  f(y)=-\frac{a\widetilde{c}(\theta-1)}{3(2\mu+\lambda)}y^\frac{\theta}{\theta-\gamma}(t),\ \
  b(t)=\int_0^t\frac{\theta-1}{2\mu+\lambda}
  \|\widetilde{F}\|_{L^\infty}y(\tau)\mathrm{d}\tau.
\end{align*}
One then deduces from \eqref{3.51*} that
\begin{equation}\label{3.52}
y'(t)\leq f(y)+b'(t).
\end{equation}
Using Lemma \ref{l2.4} with $\eta_*=\eta_1(1+t)\log^{\alpha}(e+t)$, it follows from \eqref{2.1}, \eqref{3.10}, \eqref{3.44}, and \eqref{3.47} that
\begin{align}\label{3.53}
  \|\rho\dot{\mathbf{u}}\|_{L^4(\mathbb{R}^2_+)}
   \leq
  C(\hat{\rho},M)(1+t)^4\big(\|\sqrt{\rho}\dot{\mathbf{u}}\|_{L^2(\mathbb{R}^2_+)}
  +\|\nabla\dot{\mathbf{u}}\|_{L^2(\mathbb{R}^2_+)} \big).
\end{align}
This along with \eqref{3.34}, \eqref{3.37}, Lemma \ref{l2.2}, and Lemma \ref{l2.7} yields that, for $0\leq t_1<t_2\leq \sigma(T)$,
\begin{align}\label{3.54}
|b(t_2)-b(t_1)|&\leq C
 \int_0^{\sigma(T)}\|\widetilde{F}\|_{L^\infty}\mathrm{d}t\notag\\
 &\leq C\int_0^{\sigma(T)}\|\widetilde{F}\|_{L^2}^{\frac13}\|\nabla \widetilde{F}\|_{L^4}^{\frac23}\mathrm{d}t\notag\\
 &\leq C\int_0^{\sigma(T)}\big(\|\nabla\mathbf{u}\|_{L^{2}}+\|P\|_{L^{2}}\big)^\frac13
 \|\rho\dot{\mathbf{u}}\|_{L^4}^\frac23\mathrm{d}t\notag\\
 &\leq C\int_0^{\sigma(T)}\big(\|\nabla\mathbf{u}\|_{L^{2}}+\|P\|_{L^{2}}\big)^\frac13
 \Big(\|\sqrt{\rho}\dot{\mathbf{u}}\|_{L^2}^\frac23+\|\nabla\dot{\mathbf{u}}\|_{L^2}^\frac23 \Big)\mathrm{d}t\notag\\
  &\leq
   C\sup_{t\in[0,\sigma(T)]}
   \Big(\sigma^\frac16\|\nabla\mathbf{u}\|_{L^{2}}^\frac13\Big)
  \Bigg[\bigg(\int_{0}^{\sigma(T)}\|\sqrt{\rho}\dot{\mathbf{u}}\|_{L^{2}}^{2}\mathrm{d}t\bigg)^{\frac{1}{3}}\bigg(\int_{0}^{\sigma(T)}\sigma^{-\frac14}\mathrm{d}t\bigg)^{\frac{2}{3}}
\notag\\&\quad+\bigg(\int_0^{\sigma(T)}\sigma\|\nabla\dot{\mathbf{u}}\|_{L^2}^2
\mathrm{d}t\bigg)^{\frac13}\bigg(\int_0^{\sigma(T)}
\sigma^{-\frac34}\mathrm{d}t\bigg)^{\frac23}\Bigg]\notag\\
  &\quad+C\sup_{t\in[0,\sigma(T)]}\|P\|_{L^{2}}^\frac13
  \Bigg[\bigg(\int_{0}^{\sigma(T)}\|\sqrt{\rho}\dot{\mathbf{u}}\|_{L^{2}}^{2}\mathrm{d}t\bigg)^{\frac{1}{3}}\bigg(\int_{0}^{\sigma(T)}1\mathrm{d}t\bigg)^{\frac{2}{3}}
\notag\\&\quad+\bigg(\int_0^{\sigma(T)}\sigma\|\nabla\dot{\mathbf{u}}\|_{L^2}^2
\mathrm{d}t\bigg)^{\frac13}\bigg(\int_0^{\sigma(T)}
\sigma^{-\frac12}\mathrm{d}t\bigg)^{\frac23}\Bigg]\notag\\
&\leq C_1(\hat{\rho},M,K)C_0^\frac{\alpha+2}{6(7\alpha+5)}.
\end{align}
We now apply Lemma \ref{lzlo} with
\begin{equation*}
  N_1=0, \ \
  N_0=C_1(\hat{\rho},M,K)C_0^\frac{\alpha+2}{6(7\alpha+5)}, \ \
  \xi^*=1,
\end{equation*}
where one sees that
\begin{equation*}
  f(\xi)=-\frac{a\tilde{c}(\theta-1)}{3(2\mu+\lambda)}\xi^\frac{\theta}{\theta-\gamma}\leq-N_1=0,
  \ \ \ \text{for all} \ \xi\geq \xi^*=1.
\end{equation*}
Therefore, we have that
\begin{equation}\label{3.55}
  y(t)\leq \max \big\{\hat{\rho}^\theta, 1\big\} + N_{0}\leq\hat{\rho}^\theta+C_1(\hat{\rho},M,K)C_0^\frac{\alpha+2}{6(7\alpha+5)}
  \leq \Big(\frac{3}{2}\hat{\rho}\Big)^\theta
\end{equation}
provided
\begin{equation*}
  C_0\leq \varepsilon_{4,1}\triangleq
  \min\left\{\varepsilon_3,
  \bigg(\frac{\hat{\rho}^\theta}{C_1(\hat{\rho},M,K)}\bigg)
  ^\frac{6(7\alpha+5)}{\alpha+2}\right\}.
\end{equation*}

For $\sigma(T)\leq t_1<t_2\leq T$, we first derive large-time weighted estimates.
Noting that
\begin{equation*}
\|\widetilde{F}\|_{L^{4\gamma}}\leq C\|\nabla\widetilde{F}\|_{L^{\frac{4\gamma}{2\gamma+1}}} \leq
C\|\rho\dot{\mathbf{u}}\|_{L^{\frac{4\gamma}{2\gamma+1}}}\leq
  C\|\sqrt{\rho}\|_{L^{4\gamma}}\|\sqrt{\rho}\dot{\mathbf{u}}\|_{L^2},
\end{equation*}
we deduce that
\begin{align}\label{ef2}
 \int P^2\mathrm{d}\mathbf{x}
 &=\int P\big(-F+(2\mu+\lambda)\divv\mathbf{u}\big)\mathrm{d}\mathbf{x}
 \notag\\
 &\leq
 \frac{2\mu\Lambda(2)}{3\mu+\lambda}\|P\|_{L^2}^2
 +\|P\|_{L^\frac{4\gamma}{4\gamma-1}}\|\widetilde{F}\|_{L^{4\gamma}}
 +C\|P\|_{L^2}\|\nabla\mathbf{u}\|_{L^2}
 \notag\\
 &\leq
 \frac13\|P\|_{L^2}^2
 +C\|\rho^\frac12\|_{L^2}\|\rho^{\gamma-\frac12}\|_{L^\frac{4\gamma}{2\gamma-1}}
 \|\sqrt{\rho}\|_{L^{4\gamma}}\|\sqrt{\rho}\dot{\mathbf{u}}\|_{L^2}
 +C\|P\|_{L^2}\|\nabla\mathbf{u}\|_{L^2}
 \notag\\
 &\leq
  \frac12\|P\|_{L^2}^2
  +C\big(\|\sqrt{\rho}\dot{\mathbf{u}}\|_{L^2}^2+\|\nabla\mathbf{u}\|_{L^2}^2\big).
\end{align}
Testing \eqref{3.7} with $2t(\gamma-1)P$ yields
\begin{align}\label{ef3}
 &\frac{\mathrm{d}}{\mathrm{d}t}\big(t\|P\|_{L^2}^2\big)
 +\frac{2\gamma-1}{2\mu+\lambda}t\|P\|_{L^3}^3
 \notag\\
 &\leq
 C\|P\|_{L^2}^2
 -\frac{2\gamma-1}{2\mu+\lambda}\int tFP^2\mathrm{d}\mathbf{x}
 \notag\\
 &\leq
 C\|P\|_{L^2}^2
 +\frac{2(2\gamma-1)}{3(2\mu+\lambda)}t\|P\|_{L^3}^3
+\frac{2\gamma-1}{3(2\mu+\lambda)}\bigg(4t\|\widetilde{F}\|_{L^3}^3+
4t\left(\frac{2\mu\Lambda(3)}{3\mu+\lambda}\right)^3\|P\|_{L^3}^3\bigg)
\notag\\
&\leq
C\big(\|\sqrt{\rho}\dot{\mathbf{u}}\|_{L^2}^2+\|\nabla\mathbf{u}\|_{L^2}^2\big)
+\frac{3(2\gamma-1)}{4(2\mu+\lambda)}t\|P\|_{L^3}^3
+\frac{2\gamma-1}{16\widehat{C}(2\mu+\lambda)}t\|\sqrt{\rho}\dot{\mathbf{u}}\|_{L^2}^2
   +Ct\|\nabla\mathbf{u}\|_{L^2}^6
   \notag\\
   &\quad
   +Ct\|P\|_{L^2}^2\big(\|\sqrt{\rho}\dot{\mathbf{u}}\|_{L^2}^2
   +\|\nabla\mathbf{u}\|_{L^2}^2\big),
\end{align}
where in the last inequality we have used \eqref{ef2}, \eqref{1.11}, and the following estimate (obtained as in \eqref{3.33}):
\begin{align*}
  \frac{4(2\gamma-1)}{3(2\mu+\lambda)}\|\widetilde{F}\|_{L^3}^3
  &\leq C \|\widetilde{F}\|_{L^2}\|\widetilde{F}\|_{L^4}^2\\
  &\leq C\big(\|\nabla\mathbf{u}\|_{L^2}+\|P\|_{L^2}\big)^\frac32
\|\sqrt{\rho}\|_{L^6}^\frac32\|\sqrt{\rho}\dot{\mathbf{u}}\|_{L^2}^\frac32
  \notag\\
  &\leq
 \frac{2\gamma-1}{16\widehat{C}(2\mu+\lambda)}\|\sqrt{\rho}\dot{\mathbf{u}}\|_{L^2}^2
  +C\big(\|\nabla\mathbf{u}\|_{L^2}+\|P\|_{L^2}\big)^6
  \notag\\
  &\leq
   \frac{2\gamma-1}{16\widehat{C}(2\mu+\lambda)}\|\sqrt{\rho}\dot{\mathbf{u}}\|_{L^2}^2
   +C\|\nabla\mathbf{u}\|_{L^2}^6
   +C\|P\|_{L^2}^2\big(\|\sqrt{\rho}\dot{\mathbf{u}}\|_{L^2}^2
   +\|\nabla\mathbf{u}\|_{L^2}^2\big).
\end{align*}
Proceeding as in the derivation of \eqref{3.15}, with
$\sigma$ replaced by $t$, and then adding
$\frac{4\widehat{C}(2\mu+\lambda)}{2\gamma-1}\times$\eqref{ef3}, we obtain
\begin{align}\label{ef4}
  &\frac{\mathrm{d}}{\mathrm{d}t}
  \bigg(t B(t)+\frac{4\widehat{C}(2\mu+\lambda)}{2\gamma-1}t\|P\|_{L^2}^2\bigg)+
  \frac12t\Big(\|\sqrt{\rho}\dot{\mathbf{u}}\|_{L^2}^2+\widehat{C}\|P\|_{L^3}^3\Big)
  \notag\\
&\leq
C\big(\|\sqrt{\rho}\dot{\mathbf{u}}\|_{L^2}^2+\|\nabla\mathbf{u}\|_{L^2}^2\big)
   +Ct\|\nabla\mathbf{u}\|_{L^2}^4
   +Ct\|P\|_{L^2}^2\big(\|\sqrt{\rho}\dot{\mathbf{u}}\|_{L^2}^2+\|\nabla\mathbf{u}\|_{L^2}^2\big),
   \quad \text{for } t\in (\sigma(T),T),
\end{align}
owing to
\begin{align*}
Ct\|\nabla\mathbf{u}\|_{L^3}^3+C\int t P|\nabla\mathbf{u}|^2\mathrm{d}\mathbf{x}
&\leq Ct\big(\|\nabla\mathbf{u}\|_{L^3}^3+\|P\|_{L^3}^3\big)
\notag\\
&\leq
Ct\big(\|\sqrt{\rho}\|_{L^3}^3\|\sqrt{\rho}\dot{\mathbf{u}}\|_{L^2}^3+\|P\|_{L^3}^3\big)
\notag\\
&\leq
Ct\Big(A_2^\frac12(T)\|\sqrt{\rho}\|_{L^3}^3\|\sqrt{\rho}\dot{\mathbf{u}}\|_{L^2}^2+\|P\|_{L^3}^3\Big)
\notag\\
&\leq
\frac14t\|\sqrt{\rho}\dot{\mathbf{u}}\|_{L^2}^2
 +\frac{\widehat{C}}{2}t\|P\|_{L^3}^3, \quad \text{for } t\in (\sigma(T),T).
\end{align*}
Applying Gr\"onwall's inequality to \eqref{ef4} over $(\sigma(T),T)$, it follows from \eqref{3.16} and \eqref{3.27} that
\begin{align}\label{ef5}
  \sup_{t\in[\sigma(T),T]}\big(t\|\nabla \mathbf{u}\|_{L^2}^2+t\|P\|_{L^2}^2\big)
  +\int_{\sigma(T)}^Tt\big(\|\sqrt{\rho} \dot{\mathbf{u}}\|_{L^2}^2
  +\|P\|_{L^3}^3\big)\mathrm{d}t\leq C.
\end{align}
Following the same arguments as in \eqref{ef3} and \eqref{3.33}, testing \eqref{3.7} with $3t^\frac32(\gamma-1)P^2$ and then integrating over $(\sigma(T),T)$ leads to
\begin{align*}
&\sup_{t\in[\sigma(T),T]}\big(t^\frac32\|P\|_{L^3}^3\big)
  +\int_{\sigma(T)}^Tt^\frac32\|P\|_{L^4}^4\mathrm{d}t \\
  &\leq
  C\int_{\sigma(T)}^Tt^\frac12\|P\|_{L^3}^3\mathrm{d}t
  +C\int_{\sigma(T)}^Tt^\frac32\|\widetilde{F}\|_{L^4}^4\mathrm{d}t
  \notag\\
  &\leq
  C+CA_2^\frac12(T)\sup_{t\in[\sigma(T),T]}\big(t^\frac12\|\nabla\mathbf{u}\|_{L^{2}}
+t^\frac12\|P\|_{L^{2}}\big)
  \int_{\sigma(T)}^Tt\|\sqrt{\rho} \dot{\mathbf{u}}\|_{L^2}^2\mathrm{d}t
  \leq C,
\end{align*}
which combined with the estimates as in \eqref{3.24}--\eqref{3.26} (with $\sigma^3$ replaced by $t^\frac32$) shows that
\begin{align}\label{ef6}
&\sup_{t\in[\sigma(T),T]}\big(t^\frac32\|\sqrt{\rho}\dot{\mathbf{u}}\|_{L^2}^2\big)
 +\int_{\sigma(T)}^Tt^\frac32\|\nabla\dot{\mathbf{u}}\|_{L^2}^2\mathrm{d}t
 \notag\\
 &\leq
 C+CA_2^\frac12(T)\sup_{t\in[\sigma(T),T]}\big(t^\frac12\|\nabla\mathbf{u}\|_{L^{2}}
+t^\frac12\|P\|_{L^{2}}\big)
  \int_{\sigma(T)}^Tt\|\sqrt{\rho} \dot{\mathbf{u}}\|_{L^2}^2\mathrm{d}t
  +C\int_{\sigma(T)}^Tt^\frac32\|P\|_{L^4}^4\mathrm{d}t
 \leq C.
\end{align}
Consequently, we deduce from \eqref{ef6}, \eqref{3.53}, Lemma \ref{l2.8}, Young's inequality, and \eqref{3.27} that
\begin{align}\label{3.57}
&|b(t_2)-b(t_1)|\notag\\
&\leq C
 \int_{t_1}^{t_2}\|\widetilde{F}\|_{L^\infty}\mathrm{d}t\notag\\
  &\leq C\int_{t_1}^{t_2}\|\widetilde{F}\|_{L^{\frac{52}{3}}}^{\frac{13}{16}}\|\nabla \widetilde{F}\|_{L^4}^{\frac{3}{16}}\mathrm{d}t\notag\\
   &\leq C\int_{t_1}^{t_2}\|\nabla\widetilde{F}\|_{L^{\frac{52}{29}}}^{\frac{13}{16}}\|\nabla \widetilde{F}\|_{L^4}^{\frac{3}{16}}\mathrm{d}t
   \notag\\
  &\leq C\int_{t_1}^{t_2}\|\sqrt{\rho}\|_{L^{\frac{52}{3}}}^{\frac{13}{16}}
  \|\sqrt{\rho}\dot{\mathbf{u}}\|_{L^2}^{\frac{13}{16}}(1+t)^\frac{3}{4}
  \Big(\|\sqrt{\rho}\dot{\mathbf{u}}\|_{L^2}^\frac{3}{16}+\|\nabla\dot{\mathbf{u}}\|_{L^2}^\frac{3}{16}\Big)\mathrm{d}t
  \notag\\
 &\leq
C\sup_{t\in[\sigma(T),T]}\big(\sigma^3\|\sqrt{\rho}\dot{\mathbf{u}}\|_{L^2}^2\big)
  \int_{\sigma(T)}^T\sigma\|\sqrt{\rho}\dot{\mathbf{u}}\|_{L^{2}}^{2}\mathrm{d}t
 +C\int_{\sigma(T)}^T\sigma^3\|\nabla\dot{\mathbf{u}}\|_{L^{2}}^{2}\mathrm{d}t
 +C\sup_{t\in[\sigma(T),T]}\big(t^\frac{3}{4}\|\sqrt{\rho}\dot{\mathbf{u}}\|_{L^2}\big)
 \|\rho\|_{L^\frac{26}{3}}^{\frac{13}{32}}
 \notag\\
&\quad
+C\sup_{t\in[\sigma(T),T]}\Big(t^\frac{39}{64}\|\sqrt{\rho}\dot{\mathbf{u}}\|_{L^2}^\frac{13}{16}\Big)
\bigg(\int_{\sigma(T)}^Tt^\frac32\|\nabla\dot{\mathbf{u}}\|_{L^2}^2\mathrm{d}t\bigg)^\frac{3}{32}
\bigg(\int_{t_1}^{t_2}1\mathrm{d}t\bigg)^\frac{29}{32}\|\rho\|_{L^\frac{26}{3}}^{\frac{13}{32}}
+\frac{a\tilde{c}(\theta-1)}{6(2\mu+\lambda)}(t_2-t_1)
\notag\\
&\leq
 \frac{a\tilde{c}(\theta-1)}{3(2\mu+\lambda)}(t_2-t_1)+
C_5(\hat{\rho})C_0^\frac{19-\alpha}{32(7\alpha+5)}.
\end{align}
Selecting
\begin{equation*}
  N_0=C_5(\hat{\rho})C_0^\frac{19-\alpha}{32(7\alpha+5)}, \ \
  N_1=\frac{a\tilde{c}(\theta-1)}{3(2\mu+\lambda)}, \ \
  \xi^*=1,
\end{equation*}
then one gets that
\begin{equation*}
  f(\xi)=-\frac{a\tilde{c}(\theta-1)}{3(2\mu+\lambda)}\xi^\frac{\theta}{\theta-\gamma}\leq-N_1,
  \ \ \ \text{for all} \ \xi\geq \xi^*=1.
\end{equation*}
It follows from Lemma \ref{lzlo} that
\begin{equation}\label{3.58}
y(t)\leq \max \big\{\hat{\rho}^\theta, 1\big\} + N_{0}\leq\hat{\rho}^\theta+C_5(\hat{\rho})C_0^\frac{19-\alpha}{32(7\alpha+5)}
  \leq \Big(\frac{3}{2}\hat{\rho}\Big)^\theta
\end{equation}
provided
\begin{equation*}
  C_0\leq \varepsilon_{4,2}\triangleq
  \min\left\{\varepsilon_3,\bigg(\frac{\hat{\rho}^\theta}
  {C_5(\hat{\rho})}\bigg)^\frac{32(7\alpha+5)}{19-\alpha}\right\}.
\end{equation*}
Consequently, the desired \eqref{3.49} follows from \eqref{3.55} and \eqref{3.58}
as long as $C_0\leq\varepsilon_4\triangleq\min\{\varepsilon_{4,1},\varepsilon_{4,2}\}$.
\end{proof}

Now we are ready to prove Proposition $\ref{p3.1}$.

\begin{proof}[Proof of Proposition \ref{p3.1}.]
Proposition \ref{p3.1} follows from Lemmas \ref{l3.3}, \ref{l3.4}, and \ref{l3.6} if we select $\varepsilon=\varepsilon_4$.
\end{proof}

\section{Proof of Theorem \ref{t1.1}}\label{sec4}

With the \textit{a priori} estimates established in Section \ref{sec3}, we are now in a position to prove Theorem \ref{t1.1}.

\begin{proof}[Proof of Theorem \ref{t1.1}.]
Let $(\rho_0, \mathbf{u}_0)$ be initial data as described in the theorem.
For $\epsilon>0$, let $J_\epsilon=J_\epsilon(\mathbf{x})$ be the standard mollifier, and
define the approximate initial data $(\rho_0^\epsilon, \mathbf{u}_0^\epsilon)$:
\begin{align*}
\rho_0^\epsilon&=[J_\epsilon\ast(\rho_0\mathbf{1}_{\mathbb{R}^2_+})]
\mathbf{1}_{\mathbb{R}^2_+}+\phi\epsilon,  \quad \text{with} \
\phi=\phi(\epsilon, \mathbf{x})\triangleq \psi_{1/\epsilon}(\mathbf{x})+(1-\psi_{1/\epsilon}(\mathbf{x}))e^{-|\mathbf{x}|^2}\leq 1,
\end{align*}
where the cut-off function $\psi$ is given in \eqref{3.45},
and $\mathbf{u}_0^\epsilon$ is the unique smooth solution to the elliptic equation
\begin{equation*}
\begin{cases}
\Delta \mathbf{u}_0^\epsilon = \Delta(J_\epsilon\ast\mathbf{u}_0), & \mathbf{x}\in \mathbb{R}^2_+, \\
\mathbf{u}_0^\epsilon =\mathbf{0}, & x_2=0.
\end{cases}
\end{equation*}
Then we have
\begin{align*}
\bar{x}^\alpha\rho_0^\epsilon\in L^1,\ \
(\rho_0^\epsilon-\phi\epsilon)\in H^2,\ \
\inf_{\mathbf{x}\in\widetilde{B}_{1/\epsilon}}\{\rho_0^\epsilon(\mathbf{x})\}\geq \epsilon, \ \ \mathbf{u}_0^\epsilon\in D^2\cap D_0^1,\ \
\sqrt{\rho_0^\epsilon}\mathbf{u}_0^\epsilon\in L^2.
\end{align*}

According to Lemma \ref{l2.1}, there exists a time $T_*>0$ such that the problem \eqref{a1}--\eqref{a4}
with initial data $(\rho_0^\epsilon, \mathbf{u}_0^\epsilon)$
admits a unique strong solution $(\rho^\epsilon,{\bf u}^\epsilon)$ on $\mathbb{R}^2_+\times(0,T_*]$ satisfying
\begin{equation*}
\bar{x}^\alpha\rho^\epsilon \in L^\infty(0,T; L^{1}), \ \
\rho^\epsilon-\phi\epsilon\in C([0,T_*];H^2),\
{\bf u}^\epsilon\in C([0,T_*];D^2\cap D^1), \ \inf\limits_{({\bf x},t)\in\widetilde{B}_{1/\epsilon}\times [0,T_*]}\rho^\epsilon({\bf x},t)>0.
\end{equation*}
It follows from \eqref{1.12} and \eqref{3.1}--\eqref{3.3} that
\begin{equation*}
  A_1(0)=A_2(0)=0, \ \ A_3(0)\leq M^2\leq K,
\end{equation*}
which implies that there is $T_1\in(0,T_*]$ such that \eqref{3.4} holds for $T=T_1$.
Set
\begin{equation}\label{4.1}
  T^*=\sup\{T\,| \,\eqref{3.4} \ \text{holds}\}.
\end{equation}
Obviously $T^*\ge T_1>0$. We claim
\begin{equation}\label{4.2}
  T^*=\infty.
\end{equation}
Otherwise, Lemma \ref{l3.5} gives that, for $0<T< T^*$,
\begin{align*}
  \sup_{t\in[0,T]}\|\bar{x}^\alpha\rho^\epsilon(t)\|_{L^1}\leq C(T),
\end{align*}
which implies the tightness of $\{\rho^\epsilon\}$ at infinity: for any $\delta>0$ there exists $R>0$ such that
\begin{equation}\label{www}
  \sup_{\epsilon>0}\sup_{t\in[0,T]}\int_{\mathbb{R}^2_+\backslash \widetilde{B}_R}\rho^\epsilon(t,\mathbf{x})\mathrm{d}\mathbf{x}\leq \delta.
\end{equation}
Note that Lemmas \ref{l3.3}, \ref{l3.4}, and \ref{l3.6} hold independently of the lower bound of initial density, the time of existence, and the parameter $\epsilon$.
From Proposition \ref{p3.1}, one can deduce that the improved bounds in \eqref{3.5} are valid for all $0<T< T^*$ provided $C_0\leq \varepsilon$.
These uniform estimates allow us to take the limit as $t\rightarrow T^*$. By the weak lower semicontinuity of the norms, we obtain a solution at time $T^*$ with regularity sufficient to serve as initial data. Applying Lemma \ref{l2.1} to this data yields an extension of solutions until some
$T^{**}>T^*$ such that \eqref{3.4} holds for any $0< T<T^{**}$, which contradicts \eqref{4.1}. Hence, \eqref{4.2} is true.

For any fixed $\tau$ and $T$ with $0<\tau<T<\infty$, it follows from
Section~\ref{sec3} that the approximate solutions
$(\rho^\epsilon,\mathbf u^\epsilon)$ satisfy the uniform bounds
\[
\begin{cases}
  \rho^\epsilon \in L^\infty(0,T;L^1\cap L^\theta),\qquad
  \sqrt{\rho^\epsilon}\mathbf u^\epsilon \in L^\infty(0,T;L^2), \\[1mm]
  \nabla\mathbf u^\epsilon\in L^\infty(\tau,T;L^2)\cap
  L^2(0,T;L^2),\qquad
  \sqrt{\rho^\epsilon}\dot{\mathbf u}^\epsilon
  \in L^\infty(\tau,T;L^2),
\end{cases}
\]
where $\theta>20\gamma$. Following an argument analogous to that in
\cite{LL14}, we also have
\begin{equation}\label{x*}
  \lim_{\epsilon\rightarrow0}\|\nabla\mathbf u_0^\epsilon-\nabla\mathbf u_0\|_{L^2}
  + \lim_{\epsilon\rightarrow0}
  \|\sqrt{\rho_0^\epsilon}\mathbf u_0^\epsilon-\sqrt{\rho_0}\mathbf u_0\|_{L^2}=0.
\end{equation}
By the above bounds and the weak compactness of bounded sets in Lebesgue spaces,
there exist functions $\rho\ge0$ and $\mathbf u$ such that, up to a subsequence,
\[
  \rho^\epsilon \rightharpoonup^\ast \rho
  \quad\text{in }L^\infty(0,T;L^1\cap L^\theta),
  \qquad
  \mathbf u^\epsilon\rightharpoonup \mathbf u
  \quad\text{in }L^2(0,T;D^1_0).
\]
Moreover, since $\gamma>1$ in two dimensions and the present density
class is stronger than the finite-energy class, the compactness theory of
Lions--Feireisl, applied locally on each half-ball
$\widetilde B_R=B_R\cap\mathbb R^2_+$, yields
\[
  \rho^\epsilon\to\rho
  \quad\text{a.e. in }(0,T)\times\mathbb R^2_+ .
\]
Hence, by Vitali's theorem and the uniform $L^\theta$ bound,
\[
  \rho^\epsilon\to\rho
  \quad\text{strongly in }L^q_{\loc}\big((0,T)\times\mathbb R^2_+\big),
  \qquad 1\le q<\theta,
\]
and consequently
\[
  P(\rho^\epsilon)\to P(\rho)
  \quad\text{strongly in }L^q_{\loc}\big((0,T)\times\mathbb R^2_+\big),
  \qquad 1\le q<\theta/\gamma.
\]
Combining this local convergence with the uniform weighted estimate for
$\bar x^\alpha\rho^\epsilon$ gives the corresponding strong convergence on the
whole half-plane. Furthermore, the standard momentum compactness argument
associated with the Lions--Feireisl theory (see \cite{PL98,NS04}) shows that
\[
  \rho^\epsilon\mathbf u^\epsilon\rightharpoonup \rho\mathbf u,
  \qquad
  \rho^\epsilon\mathbf u^\epsilon\otimes\mathbf u^\epsilon
  \rightharpoonup \rho\mathbf u\otimes\mathbf u
  \quad\text{in }\mathcal D'((0,T)\times\mathbb R^2_+).
\]
Passing to the limit in the weak formulations is now standard. The pressure term
is identified by the strong convergence of $P(\rho^\epsilon)$, and the convective
term by the preceding momentum compactness. Therefore, the limit
$(\rho,\mathbf u)$ is in fact a weak solution in the sense of Definition \ref{d1.1}
satisfying \eqref{1.14} on $\mathbb{R}^2_+\times(0,T]$
for any $0<T<T^*=\infty$.
\end{proof}

\section*{Conflict of interest}
The authors declare that they have no conflict of interest.

\section*{Data availability}
No data was used for the research described in the article.

\end{document}